\documentclass[11pt]{article}
\usepackage{tikz}
\usepackage[utf8]{inputenc}
\usepackage[T1]{fontenc}
\usepackage{graphicx}
\usepackage{xcolor}
\usepackage[english]{babel}
\usepackage{hyperref}
\usepackage{fancyhdr}
\usepackage{amsmath, amstext, amsopn, amssymb,amsthm}
\usepackage{thmtools}
\usepackage{mathtools}
\usepackage{stmaryrd}
\SetSymbolFont{stmry}{bold}{U}{stmry}{m}{n}
\usepackage{bm}
\usepackage{bbm}
\usepackage[shortlabels, inline]{enumitem}
\usepackage[left=3.5cm,right=3.5cm,top=3.5cm,bottom=3.5cm]{geometry}
\usepackage{microtype}

\usepackage{csquotes}

\addto\extrasenglish{

}
\hypersetup{
	colorlinks   = true,
	urlcolor     = blue,
	linkcolor    = blue,
	citecolor    = blue
}

\DeclareMathOperator*{\argmin}{arg\,min}
\DeclareMathOperator*{\argmax}{arg\,max}
\DeclareMathOperator{\supp}{supp}

\DeclarePairedDelimiterX{\inner}[2]{\langle}{\rangle}{{#1},{#2}}

\newcommand{\R}{\mathbb{R}}
\newcommand{\N}{\mathbb{N}}

\newcommand{\eps}{\varepsilon}

\makeatletter
\newcommand*{\bigcdot}{}%
\DeclareRobustCommand*{\bigcdot}{%
	\mathbin{\mathpalette\bigcdot@{}}%
}
\newcommand*{\bigcdot@scalefactor}{.5}
\newcommand*{\bigcdot@widthfactor}{1.15}
\newcommand*{\bigcdot@}[2]{%
	\sbox0{$#1\vcenter{}$}%
	\sbox2{$#1\cdot\m@th$}%
	\hbox to \bigcdot@widthfactor\wd2{%
		\hfil
		\raise\ht0\hbox{%
			\scalebox{\bigcdot@scalefactor}{%
				\lower\ht0\hbox{$#1\bullet\m@th$}%
			}%
		}%
		\hfil
	}%
}
\makeatother

\newcommand{\restr}[2]{{\left.\kern-\nulldelimiterspace #1 \vphantom{\big|} \right|_{#2}}}

\makeatletter
\newcommand*{\transp}{%
	{\mathpalette\@transpose{}}%
}
\newcommand*{\@transpose}[2]{%
	\raisebox{\depth}{$\m@th#1\intercal$}%
}
\makeatother

\definecolor{darkred}{HTML}{880808}

\theoremstyle{plain}
\newtheorem{sectioncount}{xxxxxxx}[section]
\newtheorem{theorem}[sectioncount]{Theorem}
\newtheorem{question}[sectioncount]{Question}

\newtheorem{lemma}[sectioncount]{Lemma}
\newtheorem{corollary}[sectioncount]{Corollary}
\theoremstyle{definition}

\newtheorem{remark}[sectioncount]{Remark}

\newtheorem{assumptionXXX}{Assumption}

\newtheorem{propertyXXX}{Property}

\usepackage[numbers,sort&compress]{natbib} 
\begin{document}

\title{On the breakdown point of transport-based quantiles}
\author{Marco Avella-Medina\thanks{Department of Statistics, Columbia University. Email: marco.avella@columbia.edu} \and Alberto González-Sanz\thanks{Department of Statistics, Columbia University. Email: alberto.gonzalezsanz@columbia.edu}}
\date{}
\maketitle

\begin{abstract}
Recent work has used optimal transport ideas to generalize the notion of (center-outward) quantiles to dimension $d\geq 2$. We study the robustness properties of these transport-based quantiles by deriving their breakdown point, roughly, the smallest amount of contamination required  to make these quantiles take arbitrarily aberrant values. We prove  that the  transport median defined in Chernozhukov et al.~(2017) and Hallin et al.~(2021)
has breakdown point of $1/2$. Moreover,  a point in the  transport depth contour of order $\tau\in [0,1/2]$ has breakdown point of $\tau$. This shows  that  the multivariate transport depth shares the same breakdown properties as its univariate counterpart. Our proof relies on a general argument connecting the breakdown point of transport maps evaluated at a point to the  Tukey depth of that point in the reference measure. 
\end{abstract}

\noindent \textit{Keywords}: Center-outward quantiles,  contamination model, multivariate medians,  optimal transport, robustness, Tukey depth.

\noindent \textit{MSC2020 subject classifications. Primary: 62G35, 62G30}.

\section{Introduction}
 
Recently, \cite{ChernozhukovEtAl.17.AOS,HallinDelBarrioCuestaAlbertosMatran.21} proposed a new notion of multivariate quantiles based on mass transportation.  They naturally lead to multivariate definitions of ranks and signs  that are distribution free and can be used to construct  statistically efficient nonparametric tests \cite{HallinDelBarrioCuestaAlbertosMatran.21, DebSen.19, Shietal.2022,deb:bhattacharya:sen2025}.  The idea behind transport quantiles is to preserve two fundamental properties of the quantile function \( Q_X \) of a univariate random variable \( X\sim P \):  it pushes the uniform (Lebesgue) measure on $[0,1]$ forward to $P$; and it is monotone, i.e., it is the derivative of a convex function. The transport map  provides a  natural multivariate analogue construction. Given a random vector \( X \sim P \) and a {\it reference  measure} \( \mu \), the transport-based quantile function is the unique (a.s.\ defined) gradient of a convex function \( \mathbf{Q}_{\pm} \) pushing $\mu$ forward to $P$. The existence and uniqueness of $\mathbf{Q}_{\pm}$ is guaranteed by McCann's Theorem \cite{McCann.95}.  

The univariate empirical median is one the most fundamental examples of a robust statistic. It estimates consistently the population median and it is highly robust in the presence of outliers. In particular, the minimum fraction of  {outlying} sample values required for the median to take arbitrarily large values is $1/2$. In other words, its breakdown point is $1/2$. In this paper we show that the breakdown point of the multivariate transport-based median is also $1/2$. More generally, we show that the transport-based depth contour of order $\tau\in [0,1/2]$ has breakdown point $\tau$ for natural choices of reference measures. This general result is analogous to the  breakdown point of the empirical univariate quantile of order \(q=k/n \in (0,1)\) which is exactly 
\( \min(k,n+1-k)/n\).  Note  that the breakdown points that we find are dimension independent and they are attained by computationally efficient estimators, unlike many classical high breakdown methods in the robust statistics literature \cite{maronna1976, davies1987,donoho:gasko1992,tyler1994,zuo2003}.  { Anecdotally, this high breakdown point result for the transport-based median came as a surprise to us, as our initial intuition was that the breakdown point was low and  we were trying to prove it. }

The proof of the breakdown point of the transport-based quantile contours relies on an unexpected relation between the optimal transport problem and the Tukey depth.  We show that the breakdown point of transport maps evaluated at a point of the support of the reference measure is equal to the Tukey depth of that point in the reference measure. Therefore in order to maximize the breakdown point of transport quantiles, one has to choose reference measures $\mu$ that attain the maximal Tukey depth of $1/2$. This is the case of natural choices of $\mu$; c.f. \autoref{rem:ref_measure} and \autoref{rem:reference_measure2}.

{Our breakdown point results cover both the finite sample breakdown point of transport-based quantiles as well as the asymptotic breakdown point \cite{donoho:huber1983,hampel1971}. These quantiles correspond respectively to a discrete and continuous transport map. After completing the first version of this manuscript, the parallel and independent work of \cite{paindaveine:passeggeri2024} was posted on the arXiv. While the core results are very similar, they are complementary as the authors only considered the semi-discrete transport motivated by the transport-based quantiles  defined in \cite{GhosalSen.19}  {whereas} we focused on the empirical quantiles defined in \cite{ChernozhukovEtAl.17.AOS, HallinDelBarrioCuestaAlbertosMatran.21} as well their natural population counterparts. Our proof is based on a geometric argument while the proofs in \cite{paindaveine:passeggeri2024} are more analytical. Our arguments can also be adapted to provide an alternative proof of the breakdown point of the semi-discrete transport problem presented in \cite{paindaveine:passeggeri2024}.  }  {We further note that \cite{paindaveine:passeggeri2024} also suggest a multivariate trimmed means construction based semi-discrete optimal transport. In particular, the authors propose to trim the reference measure according to its Tukey depth. This idea can also be applied to the discrete-discrete case using the results of our paper.}

\subsection{Breakdown point of the optimal transport map}

It turns out that studying the finite sample breakdown point of empirical transport-based quantiles, is equivalent to answering the following question, which itself is of independent interest:

\begin{question}\label{question}
   {\it  Given a  sample $X_1, \dots, X_n$, what  is the minimum fraction of bad sample values required to make the optimal transport map from   $u_1, \dots, u_n$ to $X_1, \dots, X_n$ evaluated at a particular  $u_i$  arbitrarily aberrant?}
\end{question}

%
Before delving into details, let us recall the definition of the optimal transport problem for discrete measures. We refer to \cite{Villani.03} for an in-depth exposition on the topic.  Let $X^{(n)}=\{X_1, \dots, X_n\}$ and 
$u^{(n)}=\{u_1, \dots, u_n\}$ be two finite sequence points in $\R^d $. We call $X^{(n)}$ and $u^{(n)}$   {\it target} and  {\it reference} samples.

We denote by $\Gamma(u^{(n)},X^{(n)} )$ the set of all of bijective mappings  $T: u^{(n)}\to X^{(n)}$. An {\it empirical optimal transport map} $\hat{T} =\hat{T}_{X^{(n)}}$ is any solution of the (quadratic) optimal transport problem  
\begin{equation}
    \label{DiscreteOT}
    \min_{T\in \Gamma(u^{(n)},X^{(n)} )}\frac{1}{n}\sum_{i=1}^n  \|T(u_i)-u_i\|^2. 
\end{equation}
Answering Question~\ref{question} is equivalent to determining the  {\it finite sample breakdown point} of the estimator  $\hat{T}(u_i)$, i.e.,  
 $${\rm BP}(\hat{T}(u_i))=\frac{1}{n} \min \left\{ \ell \in \{1, \dots, n\}: \ \sup_{Z^{(n)}\in \mathcal{Q}_{\ell,n}}  \|\hat{T}(u_i)-\hat{T}_{Z^{(n)}}(u_i)\|  = \infty\right \},
$$
where $\mathcal{Q}_{\ell,n}$ is the set of all sets $Z^{(n)}=\{Z_1, \dots Z_n\}\subset  \R^d $,  sharing  {at least} $n-\ell$ elements with $X^{(n)}$.  { Since the transport problem \eqref{DiscreteOT} minimizes a square loss, one could expect the resulting map to not be robust. Indeed, the square loss is well-known to lead to non-robust estimators in various settings such as location estimation, regression and PCA \cite{hampeletal1986,huber:ronchetti2009,maronnaetal2019}. In particular, one could expect the finite sample breakdown point to be $1/n$.  We will see that this is not at all the case {.} }

The breakdown point is one of the simplest and most popular tools for quantifying  robustness in statistics. The finite sample definition given above was introduced by \cite{donoho:huber1983} in the context of univariate location estimators. It was quickly extended to various multivariate settings in the robust statistics literature e.g. \cite{rousseeuw:yohai1984,rousseeuw1984,davies1987,yohai1987}. The notion of breakdown point was originally introduced using an asymptotic formulation in \cite{hampel1968,hampel1971}. We will also introduce a definition of breakdown of transport maps for population measures akin to the asymptotic formulation of the breakdown point in Section 2.

The median and more generally quantiles are fundamental statistics. However, the definitions of  median and quantiles, based on the canonical order of the real line, do not naturally extend to multivariate data. In higher dimensions, statistical depth theory offers an alternative. The literature provides various notions of depth, such as Tukey depth, spatial depth, integrated depth, or lens depth (see \cite{Mosler.2013} and references therein). The most prominent example is undoubtedly the Tukey (or halfspace) depth \cite{Tukey.1975}.  Recall that the Tukey  depth of a point $x\in \R^d$  with respect to  the set $X^{(n)}$ is defined as 
\[
\mathrm{TD}(x; X^{(n)}) = \min_{v\in \mathcal{S}^{d-1}} \left\{ \frac{1}{n} \sum_{i=1}^{n} \mathbf{1}_{\left( \langle v, X_i - x \rangle \geq 0 \right)} \right\},
\]
where $\mathcal{S}^{d-1}=\{v {\in \R^d}: \|v\| = 1\}$ denotes the unit sphere and  $\mathbf{1}_A$ is the  indicator function of the event $A$. 
It takes values in the interval \([0,1/2]\), where the observation with maximum depth  is called the Tukey median and is considered the most central, while observations with depth close to $0$ are regarded as more extreme. The breakdown point of the Tukey median is \(1/3\) for halfspace symmetric distributions\footnote{ {A distribution $\mu$ is halfspace symmetric about $\theta$ if $\mu(H)\ge 1/2$ for every closed halfspace $H$ containing $\theta$   (see \cite[p.~464]{zuo:serfling2000}). Note also that the finite sample breakdown point in \cite{donoho:gasko1992} is defined by adding outlying points instead of replacing observed data points with outlying ones. More precisely, they compute the number  $m^*$ of points that  need to be added to a sample of size $n$  in order to break the Tukey median, so their breakdown point is then $m^*/(m^*+n)$.}  } and is at least $1/(d+1)$ for general configurations of points \cite{donoho:gasko1992}. 

In this work we answer the motivating Question \ref{question} by finding an unexpected relation between the breakdown point of the optimal transport map and the Tukey depth. The following main result assumes the reference sequence of points is in \emph{general position}, i.e., no more than $d$ points lie in any $(d-1)$-dimensional affine subspace. 
\begin{theorem}    \label{Theo:mainDiscrete}
    Let $\hat{T}$ be a solution to \eqref{DiscreteOT} and  assume $u^{(n)}$ is in general position. Then, 
    $${\rm BP}(\hat{T}(u_i))\in \left[\mathrm{TD}(u_i; u^{(n)})-\frac{d-1}{n},\mathrm{TD}(u_i;u^{(n)})\right].
$$
\end{theorem}
The upper bounds stated in \autoref{Theo:mainDiscrete} are sharp as discussed in \autoref{Rem:sharpness}. 
In \autoref{Sec:MainResultsCont}, we provide a slightly more general version of  {\autoref{Theo:mainDiscrete}} as well as an analogous statement for the transport map of continuous measures.  The latter can be seen as Hampel's asymptotic breakdown point \cite{hampel1971}.

\subsection{Breakdown points of transport-based medians and contours.}
The {\it empirical transport-based quantile function} $\mathbf{Q}_{ \pm }^{(n)}$ of a sample $X^{(n)}$ of $P$ is any mapping solving \eqref{DiscreteOT}, while $\mathbf{F}_{ \pm }^{(n)}=(\mathbf{Q}_{ \pm }^{(n)})^{-1}:X^{(n)}\to u^{(n)}$ is the {\it empirical transport-based distribution function.} 
Note that these definitions depend on the choice of $u^{(n)}$. In practice, one could take a random sample from a population reference measure $\mu$. Deterministic choices include taking a    {\it ``spherical regular grid''} \cite{HallinDelBarrioCuestaAlbertosMatran.21}  or  a Halton sequence in a hypercube   \cite{DebSen.19}.

The transport-based depth $D_\pm(X_i)$ of an observation $X_i$ was defined in \cite{ChernozhukovEtAl.17.AOS} as the Tukey depth of $\mathbf{F}_{ \pm }^{(n)} (X_i)$ with respect to  $ u^{(n)}$, i.e.,     $$ {\rm D}_\pm^{(n)}(X_i)={\rm TD}( \mathbf{F}_{ \pm }^{(n)} (X_i); u^{(n)}). $$ 
So far, the depth function is defined on $X^{(n)}$. In order to extend its definition to the whole $\R^d$, 
 \cite{HallinDelBarrioCuestaAlbertosMatran.21}   proposed to extend 
 $\mathbf{F}_{ \pm }^{(n)}$ by considering  a  continuous cyclically monotone interpolation, i.e.,  by choosing  a convex and differentiable function $\psi_n=\psi_{n}(\cdot; X^{(n)}):\R^d\to \R$ such that $\nabla\psi_n(X_i)=\mathbf{F}_{ \pm }^{(n)}(X_i)$ for all $i=1, \dots, n$. Equipped with this extension one can define the {\it empirical transport-based depth function} of a point $x\in \R^d$ as 
 $$ {\rm D}_\pm^{(n)}(x)={\rm D}_\pm^{(n)}(x;X^{(n)}):={\rm TD}( \nabla\psi_n (x); u^{(n)}) $$ 
 as well as the empirical  {\it empirical transport-based depth contours} 
$$ {\cal  C}_\pm^{(n)}(\tau)=\mathcal{C}_\pm^{(n)}(\tau; X^{(n)}):=\{y\in \R^d: \ {\rm D}_\pm^{(n)}(y)=\tau \}, \quad \tau\in [0,1/2]. $$
An {\it empirical  transport-based median} $m_\pm^{(n)}$ of  $X^{(n)}$ is a deepest point   in the sense 
$$ m_\pm^{(n)}\in \argmax_{x\in \R^d}  {\rm D}_\pm^{(n)}(x) .  $$

Our main result  \autoref{Theo:mainDiscrete} provides asymptotically sharp  estimates of the 
finite sample breakdown point of $ {\cal  C}_\pm^{(n)}$, which is defined as
 $$
 {\rm BP}({\cal  C}_\pm^{(n)}(\tau) ):=\frac{1}{n} \min \left\{ k \in \{1, \dots, n\}: \ \sup_{Z^{(n)}\in \mathcal{Q}_{k,n}} d_H(\mathcal{C}_\pm^{(n)}(\tau; X^{(n)}), \mathcal{C}_\pm^{(n)}(\tau; Z^{(n)}))= \infty\right \}.
$$
Here $d_H$ denotes the 
{\it Hausdorff distance}, which   is defined as:
\[
d_H(A, B) = \max\left\{ \sup_{a \in A} \inf_{b \in B} \|a - b\|, \sup_{b \in B} \inf_{a \in A} \|b - a\| \right\}, \quad A, B \subset \mathbb{R}^d.
\]
Note that, as depth contours are compact sets\footnote{This follows by the vanishing at infinity property and lower semi-continuity \cite[ {Theorem~2.1}]{zuo:serfling2000}.}, the Hausdorff distance is a natural metric for measuring distances between depth contours. See for example  \cite{donoho:gasko1992,Brunel.2018} and \cite{delBarrioGonzalezHallin.2020}  in the context of Tukey and transport-based  depth contours respectively.
\begin{corollary}\label{coro:CenterOutward}
Let $X_1, \dots, X_n$ be a sequence of vectors in $\R^d$ with all of the elements being different and let $u^{(n)}$ be in general position. Set 
$
\tau^*:=\max_{u\in \R^d} {\rm TD}( u; u^{(n)}). $
Then the following holds irrespective of the dataset $X_1, \dots, X_n$,   
 \begin{enumerate}
     \item (Breakdown point of the median)  $$ {\rm BP}(m_\pm^{(n)})\in \left[\tau^*- {d}/n,\tau^*\right].  $$ 
     \item (Breakdown point of the depth  contours) For every $\tau\in [0,\tau^*)$ with  ${\cal  C}_\pm^{(n)}(\tau)\neq \emptyset$, 
 $$ {\rm BP}({\cal  C}_\pm^{(n)}(\tau) )\in [\tau- {d}/n,\tau]. $$ 
 \end{enumerate}
 \end{corollary}

 In \autoref{Sec:MainResultsCont}, we provide a slightly more general version of \autoref{coro:CenterOutward} that covers the case where the reference sequence is not assumed to be in general position. We also provide a continuous population counterpart that provides an asymptotic breakdown point in the sense of Hampel \cite{hampel1971}.

 \begin{remark}\label{rem:ref_measure}
     The choice of  $u^{(n)}$ affects the value of  $\tau^*$ which in turn affects the breakdown points of the transport-based median and depth contours. If $u^{(n)}$  is such that   $\tau^*=1/2$, the resulting median  preserves the high breakdown point of its univariate counterpart  {for a sufficiently large sample size}. However, $u^{(n)}$ should not be chosen very ``organized'' as the  general position assumption could be violated. If $u^{(n)}$  is a sample of a  distribution $\mu$ absolutely continuous  with respect to  the Lebesgue measure $\mathcal{L}_d$ and $n>d$, then the reference measure is in general position with probability $1$. Recall that $\mu$ is said to be halfspace symmetric if it attains the maximum Tukey depth of 1/2 (see e.g.~\cite{zuo:serfling2000}). Furthermore, if additionally, $\mu$ is halfspace symmetric  then $u^{(n)}$ has maximal Tukey depth of $1/2+o(1)$ for $d$ fixed as $n\to \infty$  {(cf.~\autoref{lemma:Consistency-ofTukey} below)}. 
 {     Note that  all the choices that have been suggested as reference measures for transport-based quantiles (see \cite{ChernozhukovEtAl.17.AOS,Figalli.2018,HallinDelBarrioCuestaAlbertosMatran.21,huang:sen2023,delBarrioGonzalezSanz.2024}) are halfspace symmetric.  For instance,  the uniform on the hypercube $[-1,1]^d$, isotropic Gaussians and the spherical uniform distribution   attain the maximal Tukey depth of $1/2$ (see \cite{donoho:gasko1992}).}
 \end{remark}
 {
The following lemma justifies the convergence of the empirical Tukey depth mentioned in the previous remark. The proof is quite standard and it is relegated to \autoref{Section:Proof-of-lemma-Tukey}. Similar results can be found in \cite[Eq.~(6.2)]{donoho:gasko1992} and \cite[Proposition~2.4]{masse2002} when $\mu_n$ is the (random) empirical measure. 
\begin{lemma}\label{lemma:Consistency-ofTukey}
     Let $\mu\in \mathcal{P}^{a.c.}(\R^d)$. Assume that $\mu_n $ converges in distribution to $\mu$. Then, 
     $$  \sup_{u\in \R^d}|{\rm TD}(u; \mu_n) - {\rm TD}(u; \mu)|\to 0.$$ As a consequence, if $\mu$ has  maximum Tukey depth $1/2$, then $$\lim_{n\to \infty}\sup_{u} {\rm TD}(u;\mu_n) = \frac{1}{2}. $$  
 \end{lemma}}
 \begin{remark} Part 2 of \autoref{coro:CenterOutward} can be viewed as a multivariate analogue of the well-known breakdown point of $\min\{\tau, 1-\tau\}$ for empirical $\tau$-quantiles. Note however that in the univariate case, the $\tau$-depth contour is not a standard $\tau$-quantile, but a ``center-outward quantile'' as described in \cite{HallinDelBarrioCuestaAlbertosMatran.21}. Indeed, the left to right orientation of standard quantiles is replaced by a center-outward orientation that is more amenable to be generalized to dimension $d\geq 2$. As such, the $1/4$-depth contour does not correspond to the $1/4$-quantile but to end-points of the interquantile range, i.e., the $1/4$- and $3/4$-quantiles. Therefore the probability contained inside a $1/4$-contour is $1/2$, and more generally the probability mass contained inside the region defined by a $\tau$-contour is $1-2\tau$.      
 \end{remark}

\subsection{Related work}

There is a rich literature focusing on deriving the breakdown point of robust multivariate estimators of location and designing high breakdown point methods. It has been well established that while univariate M-estimators defined by a bounded score function have a high breakdown point, many natural multivariate generalizations of them can have a dimension dependent breakdown \cite{maronna1976,donoho:huber1983,huber1984,tyler1994,maronna:yohai1995,kent:tyler1996}. 
Because the definition of the univariate median relies on the canonical ordering of the real line, it can be generalized in multiple natural ways to multivariate settings starting with $d=2$. As a result, numerous definitions of multivariate medians have been studied in the literature. Prominent examples include the geometric/spatial median \cite{koltchinskii1994,chaudhuri1996,minsker2015,minsker:strawn2024}, the Tukey median \cite{Tukey.1975,donoho:gasko1992,chen:gao:ren2018}, projection depth-based medians \cite{adrover:yohai2002,zuo2003,ramsay:durocher:leblanc2021,depersin:lecue2023}, among others. The coordinatewise median and the geometric median  \cite{koltchinskii1994,chaudhuri1996,minsker2015,minsker:strawn2024} can be computed efficiently and are known to have high breakdown point of $1/2$ \cite{lopuhaa:rousseeuw1991}. However they do not correspond to natural notions of depth  as discussed in \cite{zuo:serfling2000,ChernozhukovEtAl.17.AOS}.  Perhaps the most popular depth-based median is the Tukey median and is known to have a breakdown point of $1/3$ \cite{donoho:gasko1992}. Some other depth-based medians attain the maximal breakdown point of $1/2$ \cite{tyler1994,zuo2003}. However, we note that computing the Tukey median and high breakdown point projection-based estimators is NP hard \cite{bernholt2006} which significantly reduces their applicability for  $d$ moderately large. We note that not only do  transport-based medians have a high breakdown point, but they can also be computed in polynomial time with a computational cost of $O(n^3)$, irrespective of the dimension (cf.\  \cite[ {Remark~3.3}]{CuturiPeyre.19}),  where $n$ is the sample size. 
 Furthermore, transport-based quantiles can also naturally be used to define multivariate ranks. Such rank statistics are distribution free and have been shown to be useful for constructing statistically efficient nonparametric tests 
\cite{HallinDelBarrioCuestaAlbertosMatran.21, DebSen.19, Shietal.2022,deb:bhattacharya:sen2025}.

{Another popular notion of multivariate quantiles are the spatial or geometric quantiles \cite{chaudhuri1996}. Their breakdown point properties have been recently studied by \cite{konen:paindaveine2023,konen:paindaveine2025}. See \cite{hallin:konen2024} for an extensive discussion of the relative merits of transport-based quantiles and spatial quantiles. More general notions of multivariate quantiles have been explored in the recent work of  \cite{konen:paindaveine2022,konen2025}.}

{ As noted in \cite{ronchetti2023}, the robustness of optimal transport in the sense of robust statistics measures such as breakdown point and influence functions were unexplored.  Our results and the parallel work of \cite{paindaveine:passeggeri2024} constitute the first breakdown point analysis of transport maps.}   { Our results cover the discrete-discrete and continuous-continuous cases, whereas \cite{paindaveine:passeggeri2024} focused on the semi-discrete.   We recall that the semi-discrete (empirical) optimal transport map $\hat{T}^{SD}$ is the transport map from the reference measure $\mu$ to the empirical measure $P_n$. In the context of transport-based quantiles, this estimator has been proposed by \cite{GhosalSen.19}. For this estimator, \cite{paindaveine:passeggeri2024} showed that  ${\rm BP}(\hat{T}^{SD}(u_i))= \frac{\lceil n\cdot \mathrm{TD}(u_i; u^{(n)})  \rceil }{n} $. Hence, comparing this to \autoref{Theo:mainDiscrete}, we can affirm that the semi-discrete estimator is more robust than the discrete-discrete.  }

The influence function  approach is commonly referred to as infinitesimal robustness since it aims to quantify the impact of small amounts of contamination in the statistical functional of interest by finding a G\^ateaux derivative at a fixed perturbation $\delta_x$ \cite{hampel1974,hampeletal1986}.  In the context of transport maps, to the best of our knowledge, the form of the influence function 
$ \frac{d}{dt } \big\vert_{t=0}T_{\mu\to (1-t)P+t \delta_x} $ has not been established. The closest results  we found in the literature involve finding the derivative of 
$ \frac{d}{dt } \big\vert_{t=0}T_{\mu\to (1-t)P+t Q} $
when $Q$ is regular enough (see \cite{manole2024centrallimittheoremssmooth} for periodic measures  and \cite{gonzálezsanz2024linearizationmongeampereequationsdata} for the general setting).

\subsection{Organization of the paper}

The rest of the paper is organized as follows: In \autoref{Sect:Defintions}, we introduce the key definitions related to the population transport problem that we will need in this work.  In particular, we introduce the general population optimal transport problem, the canonical extension of transport maps, transport-based depth contours and an appropriate notion of  breakdown point.  
In \autoref{Sec:MainResultsCont}, we state the main results of the paper in the form of the breakdown point of the empirical and population transport maps evaluated at any point in the support of the reference measure. These results connect the breakdown point of the transport maps to the Tukey depth of the point in the reference measure, showing that  transport-based depth contours of order $\tau \in (0,1/2)$ have a breakdown point of $\tau$. Finally, \autoref{Sect:proofs}
 contains the proofs of all the results presented in the paper.

\section{Preliminaries}\label{Sect:Defintions}
In this section we introduce the population optimal transport problem and some important notions that we will require to state our main results. { The technical considerations discussed below  will be needed to formally define the asymptotic breakdown point of the transport-based quantiles. In particular, in order to define the transport map at any point we introduce maximal {  (cyclical)} monotone extensions which have also been considered in \cite{HallinDelBarrioCuestaAlbertosMatran.21, delBarriodoGonzalezSanzHallin.24, segers2022graphicaluniformconsistencyestimated,gonzálezsanz.2025.nonparametricvectorquantileautoregression}.}

\subsection{Notation}
The triple $(\Omega, \mathcal{A},\mathbb{P})  $ is the probability space. 
Let $\mathcal{P}(\R^d)$ be the set of Borel probability  measures over $\R^d$ and 
$ \mathcal{P}^{a.c.}(\R^d)$ 
be the set of probability measures absolutely continuous  with respect to  the $d$-dimensional Lebesgue measure $\mathcal{L}_d$. The  closure and interior of a set $A\subset \R^d$ is denoted as ${\rm cl}(A)$ and ${\rm int}(A)$, respectively, while ${\rm coh}(A)$ denotes the convex hull of $A$.  The support of a Borel probability measure $P$, denoted as ${\rm supp}(P)$,  is the smallest closed set  with ${P}$-probability $1$.  Let $2^{\R^d}=\{A: \ A\subset \R^d\}$  and  denote by ${\rm dom}(G)$ the domain of a set-valued function $G:\mathbb{R}^d\to 2^{\mathbb{R}^d}$ i.e., ${\rm dom}(G)$ is the set of all $x$ such that $G(x)\neq \emptyset$.  {The domain of a proper function $f:\R^d\to (-\infty, \infty]$ is denoted by ${\rm dom}(f)=\{x\in \R^d: f(x)\in \R\}$.} Finally,  let  $\mathbb{B}=\{x\in\mathbb{R}^d: ~\|x\|\leq 1\}$.
\subsection{Extension of monotone transport maps}\label{Sec:extension}
McCann's Theorem is one of the fundamental results in the measure transportation literature. It states that given two probability measures $\mu $ and $P$, with $ \mu \in  \mathcal{P}^{a.c.}(\R^d) $, there exists a unique a.s.\ defined gradient $\nabla \varphi=\nabla \varphi_{\mu\to P}$ of a lower semicontinuous (l.s.c.) convex function  $\varphi$  pushing $\mu$ forward to $P$ (see \cite{McCann.95}).  Moreover, if $\int \|u\|^2d\mu(u)$ and $\int \|x\|^2dP(x)$ are finite, $\nabla \varphi$ is the unique solution of the so-called Monge problem 
$$ \min_{T_\# \mu=P} \int \|T(u)-u\|^2 d\mu(u),  $$
where $(T_\# \mu)(\cdot)=\mu(T^{-1}(\cdot))$ denotes the push-forward measure. We call $\nabla \varphi$ {\it optimal transport map} even if the second order moments of $\mu$ and $P$ are not finite. 
The {\it canonical extension} of the transport map $\nabla \varphi$ is the subdifferential (see \cite{HallinDelBarrioCuestaAlbertosMatran.21,delBarriodoGonzalezSanzHallin.24,segers2022graphicaluniformconsistencyestimated,Figalli.2018,delBarrioGonzalezHallin.2020}), that is
\[
\R^d\ni u\mapsto \partial \varphi(u) = \left\{ x \in \mathbb{R}^n :
 \ \varphi(z) \geq \varphi(u) + \langle x, z - u \rangle \quad \forall z \in \mathbb{R}^n \right\}.
\]
It holds that $\{\nabla \varphi(u)\}= \partial \varphi(u)$ for $\mu$-a.e.\ $u\in \R^d$, which shows that $ \partial \varphi$ is indeed an extension of $ \nabla \varphi$.  However,  $\partial \varphi$ is no longer taking values in $\R^d$, but in $2^{\R^d}$. That is, $\partial \varphi$ is a set-valued map. 
A celebrated theorem by Rockafellar (see \cite{Rockafellar.MaximalMon}) states that subdifferentials of convex functions are {\it maximal {  cyclically} monotone}, i.e., {  for any choice of points $u_0,u_1,\dots,u_m$ (for arbitrary $m\geq 1$) and $x_i\in \partial \varphi(u_i)$, one has
\begin{equation}
    \label{MonotoneQuant}
    \langle x_1 , u_1 - u_0 \rangle +\langle x_2 , u_2- u_1 \rangle +\dots+\langle x_{m-1}, u_m - u_{m-1} \rangle+ \langle x_m, u_0 - u_m \rangle \geq 0 ,
\end{equation}
}
and its graph is not strictly contained in the graph of any  other set-valued mapping satisfying \eqref{MonotoneQuant}. 

\begin{remark}\label{remark:UniqueExtension}
   {  We note that although $\nabla \varphi$ is the unique gradient of a convex function pushing $\mu$ forward to $P$, the convex function $\varphi$ itself might not be unique (cf. \cite{McCann.95}); consequently, neither is its subdifferential $\partial \varphi$.} In \autoref{Sec:MainResultsCont} we show that the breakdown point of the canonical extension of the transport maps does not depend on the choice of the extension.    
\end{remark}

\subsection{ Population transport-based depth contours}\label{Sec:ContCOntorurs}

We call \( \mu \in \mathcal{P}^{a.c.}(\R^d)\)  the  {\it reference measure}.   The {\it transport-based quantile function} is the unique ($\mu$-a.e.\ defined) gradient of a convex function \( \mathbf{Q}_{\pm}\) pushing $\mu$ forward to $P$. The {\it transport-based distribution function} $\mathbf{F}_{ \pm }$ of $P$ is the $P$-a.e.\ defined inverse $\mathbf{F}_{ \pm } =(\mathbf{Q}_{ \pm })^{-1} $ of the transport-based quantile map. 
We recall that our goal is to find the breakdown point of  $\mathbf{Q}_{ \pm }(x)$ for $x\in \R$. Let ${\rm dom}(\mathbf{Q}_{ \pm})$ denote the domain of $\mathbf{Q}_{ \pm}$ and  note that in general  $\mathbf{Q}_{ \pm}$ is well-defined $\mu$-a.e., so that  $\mathbf{Q}_{ \pm }(x)$ might be undefined for some $x\in \R^d$.  Hence, even to guarantee that ${\rm dom}(\mathbf{Q}_{ \pm})$ contains the support of $\mu$ one needs to impose regularity conditions over $\mu$ and $P$ in order to guarantee that $\mathbf{Q}_{ \pm}$ and $\mathbf{F}_{ \pm}$ are well-defined everywhere.  A standard condition is to assume  that $\mu$ and $P$ have densities supported on a convex set locally bounded away from zero and infinite   \cite{Caffarelli.92, Figalli.17}.  However, one of the most natural choices of $\mu$, the so-called spherical uniform distribution $\mu_d$  
has  density 
$ {C_d}{\|x\|^{1-d}} $ for some dimension-dependent constant $C_d>0$ \cite{HallinDelBarrioCuestaAlbertosMatran.21}. Since $\mu_d$ 
has a singularity at zero for $d>1$, it violates the aforementioned standard regularity condition.\footnote{One can therefore not rely on this notion of regularity for $\mu_d$ in order to get well-defined quantile and distribution functions, irrespective of the regularity of $P$.  However, far away from zero, one can still show that the quantile map is well-defined for $P\in \mathcal{P}^{a.c.}(\R^d)$ and has a  density supported on a convex set locally bounded away from zero and infinite (cf. \cite{ Figalli.2018,delBarrioGonzalezHallin.2020,delBarrioGonzalezSanz.2024}). 
} In particular, the median $\mathbf{Q}_{ \pm}(0)$, might be undefined for some arbitrarily regular probability measure $P\in \mathcal{P}^{a.c.}(\R^d)$, for $d>2$. (Note that for $d=2$ the Monge–Ampère equation have a slightly nicer behavior, see \cite{Figalli.2018}.)  
In \cite{delBarriodoGonzalezSanzHallin.24} this  issue is solved by extending $\mathbf{Q}_{\pm}$ in the { ``canonical way''} introduced  for transport maps. 
Specifically, a {\it set-valued quantile function} is 
\[
\R^d\ni u\mapsto \mathbb{Q}_\pm (u):=\partial \phi(u) = \left\{ x \in \mathbb{R}^n :
 \ \phi(z) \geq \phi(u) + \langle x, z - u \rangle \quad \forall\, z \in \mathbb{R}^n \right\},
\]
where $\phi:\R^d\to (-\infty, +\infty]$ is any l.s.c.\  convex function such that $\nabla \phi(u)= \mathbf{Q}_{ \pm }(u)$ for $\mu$-a.e.\ $u\in \R^d$.  This extension allows us to define the {\it transport-based depth}  of a point $x\in \R^d$  with respect to  a probability measure $P$ as the set
 $$ {\rm D}_\pm(x)={\rm D}_\pm(x;P):={\rm TD}( \mathbb{Q}_{ \pm }^{-1} (x); \mu), $$ 
where 
$$ \mathrm{TD}(u; \mu):= \min_{v\in \mathcal{S}^{d-1}} \mu\left( \{z: \langle v, z - x \rangle \geq 0\} \right) $$
is the celebrated Tukey depth and 
$ \mathbb{Q}_{ \pm }^{-1}(u)=\{u\in \R^d:  x\in \mathbb{Q}_{ \pm }^{-1} (u)\}  $ is the set-valued inverse of $ \mathbb{Q}_{ \pm }$.  We assume throughout the rest of the paper that the reference measures used to define the quantiles are halfspace symmetric. This is a standard assumption in the literature; c.f.  \autoref{rem:ref_measure}.  
The {\it transport-based depth  contours} are defined as
$$ {\cal  C}_\pm(\tau)=\mathcal{C}_\pm(\tau; P):=\{x\in \R^d: \ {\rm D}_\pm(x)=\tau \}, \quad \tau\in [0,1/2]$$ 
and the transport-based regions as
$$\mathcal{R}_{\pm}(\tau)= \{x\in\mathbb{R}^d: \mathrm{D}_\pm (x)\geq \tau\}, \quad {\rm for} \ \tau\in (0,1/2).$$
{\it transport-based median as} 
$$m_\pm(P)=\argmax_{x\in\mathbb{R}^d}{\rm D}_\pm(x).$$

 {
We note that in the univariate case the transport-based median  coincides with the set of standard univariate medians. See \autoref{lem:univariate-median} for a formal statement.
}

\subsection{Breakdown points of population contours}\label{Sect:breakdownSet}

 We define the breakdown point of the depth contour ${\cal  C}_\pm(\tau;P )$ at   $P\in \mathcal{P}^{a.c.}(\R^d)$ as
    $${\rm BP}({\cal  C}_\pm(\tau;P )):= \inf \left\{ \varepsilon\in (0,1): \ \sup_{Q\in \mathcal{Q}_{\varepsilon}(P)} d_H(\mathcal{C}_\pm(\tau; P), \mathcal{C}_\pm(\tau; Q))= \infty\right \},
$$
 where ${\cal Q}_\eps(P)=\{ P+\eps (Q-P): \ Q\in \mathcal{P}(\R^d)\}\cap \mathcal{P}^{a.c.}(\R^d)$. Note that for $X_1,\dots,X_n\overset{iid}{\sim}P$ our results (\autoref{coro:CenterOutward} and \autoref{coro:median_contours_cont}) show that ${\rm BP}({\cal  C}_\pm^{(n)}(\tau)) \to {\rm BP}({\cal C}_{\pm}(\tau;P ))$ almost surely. One can therefore think of ${\rm BP}({\cal  C}_\pm(\tau;P ))$ as the asymptotic breakdown point of the $\tau$-depth contour at the distribution $P$.  This is  similar to the original definition of breakdown of Hampel \cite{hampel1971}.
 
Clearly our definition of the contour breakdown point  only makes sense if  the contour sets are compact. This is automatically guaranteed by 
\autoref{Lemma:TukeyPositiveCont}, the compactness of the Tukey depth's   {contours} and the definition of $\cal C_{\pm}(\tau).$

\begin{lemma}
\label{Lemma:TukeyPositiveCont} 
    Set $\mu,P\in \mathcal{P}^{a.c.}(\R^d)$. Let $\partial \varphi$ be any maximal {  cyclical} monotone extension of  the optimal transport map $\nabla \varphi$ from $\mu$ to $P$.  Then it holds that 
    \begin{equation}
        \label{eq:TukeyPositiveCont}
        \{u\in \R^d: 
    {\rm TD}(u;\mu)>0\}\subset {\rm int}\left({\rm dom}(\partial \varphi)\right).
    \end{equation}
    As a consequence, for any compact set $K\subset \{u\in \R^d: 
    {\rm TD}(u;\mu)>0\}$ it holds that 
    $\partial \varphi(K)=\bigcup_{u\in K} \partial \varphi(u)$ is  compact and nonempty.
\end{lemma}


\begin{proof}

We show first that
\begin{equation}
    \label{eq:TukeyCont1}
    \{u\in \R^d: 
    {\rm TD}(u;\mu)>0\}\subset {\rm int}({\rm coh} (\supp (\mu))),
\end{equation}
 {where for a set $A\subset \R^d$,  $ {\rm coh} (A)$ denotes its convex hull. We fix $u\notin {\rm int}({\rm coh} (\supp (\mu)))$. Then by the separation theorem (see e.g.,~\cite[Theorem~11.2]{Rockafellar.70}) there exists a hyperplane $ H=\{ x\in \R^d : \langle x, a\rangle=b\} $ such that $u\in H$ and  $$  {\rm int}({\rm coh} (\supp (\mu))) \subset {\rm int } (H_-)=\{ x\in \R^d : \langle x, a\rangle < b\}. $$ 
As a consequence, it follows that 
\begin{equation}
    \label{eq:TukeyCont0}
    \supp (\mu)\subset {\rm cl}({\rm coh} (\supp (\mu)))= {\rm cl}({\rm int}({\rm coh} (\supp (\mu)))) \subset H_- =\{ x\in \R^d : \langle x, a\rangle \leq  b\},
\end{equation}
where the first equality follows from \cite[Theorem~6.3]{Rockafellar.70} applied to the convex set \({\rm coh} (\supp (\mu))\). (Note that the relative interior of \({\rm coh} (\supp (\mu))\) coincides with its interior; otherwise, we would have \(\mu({\rm coh} (\supp (\mu))) = 0\).) Define  $H_+=\{ x\in \R^d : \langle x, a\rangle \geq b\}$. Since the hyperplane $H$  --- the boundary of $H_+$ --- is negligible for the Lebesgue measure and  $\mu\in \mathcal{P}^{a.c.}(\R^d)$,  \eqref{eq:TukeyCont0} implies that   
$$ \mu( H_+) = \mu ({\rm int}(H_+))= \mu (\R^d\setminus H_-) \leq \mu (\R^d\setminus {\rm supp}(\mu) )=0 .  $$
As $u $ belongs to the boundary  of the closed halfspace $H_+$, then $ {\rm TD}(u;\mu)=0$. We have established that $\{u\in \R^d: 
    {\rm TD}(u;\mu)=0\}\supset \mathbb{R}^d\setminus{\rm int}({\rm coh} (\supp (\mu)))$ and hence  \eqref{eq:TukeyCont1} follows. 
}
 Next, 
     as $\nabla \varphi$ pushes $\mu$  forward to $P$, it holds that  $\mu({\rm cl}({{\rm dom}(\partial \varphi)}))=1$. Then, by definition of ${\rm supp}(\mu)$, it holds that 
     ${\rm supp}(\mu)\subset {\rm cl}({{\rm dom}(\partial \varphi)})$.   
By Theorem~12.41 and the first display after in \cite{RockafellarWets.98}, it follows that ${\rm cl}({{\rm dom}(\partial \varphi)})$  is convex and 
$${\rm int}({{\rm dom}(\partial \varphi)}) = {\rm int}({\rm cl}({{\rm dom}(\partial \varphi)})) \supset {\rm int}({\rm coh}({\rm supp}(\mu))). $$
From this display and  
\eqref{eq:TukeyCont1}, we derive \eqref{eq:TukeyPositiveCont}. 
   The final conclusion follows from Theorems 23.4 and 24.7 in \cite{Rockafellar.70}. 
\end{proof}

Our proof of the breakdown of ${\cal  C}_\pm(\tau;P )$ relies on breaking the extended transport map  $\partial \varphi$.  This is the key workhorse of our argument. Since $\partial \varphi(u)\neq \emptyset$ might not be unique (cf. \autoref{remark:UniqueExtension}), we use a selection functional to define its breakdown point as follows 
$${\rm BP}(\partial \varphi(u), P)= \min \left\{ \eps\in (0,1): \sup_{\varphi_{\mu\to \nu}\in \Gamma_\eps(P)} \inf_{\mathrm{v}\in \partial \varphi_{\mu\to \nu} (u)} \|\mathrm{v}\| = \infty\right \},
$$
where $\Gamma_\eps(P)$ is the set of l.s.c.\ convex functions $\varphi_{\mu\to \nu}$ such that $\nabla \varphi_{\mu\to \nu}$ pushes $\mu$ forward to  $\nu\in {\cal Q}_\eps(P)$. This selection functional defines the breakdown point as the smallest $\varepsilon$ such that  the smallest element of $\partial\varphi_{\mu\to\nu}(u)$ explodes.
It turns out that the stronger requirement that all the elements $\partial\varphi_{\mu\to\nu}(u)$ explode leads to the same breakdown point. See \autoref{rem:anotherBP}.

\section{Main results}\label{Sec:MainResultsCont}

\subsection{Breakdown points of the optimal transport map}

This section contains a generalization of \autoref{Theo:mainDiscrete}, valid for reference samples in non-general position,  and  its population version.  We define the {\it lower Tukey depth} of $u$ with respect to the dataset $u^{(n)}$ as 
\[
\mathrm{TD}^-(u; u^{(n)}):= \frac{n+1}{n}-\max_{v\in \mathcal{S}^{d-1}} \left\{ \frac{1}{n} \sum_{i=1}^{n} \mathbf{1}_{\left( \langle v, u_i - u \rangle \geq 0 \right)} \right\}.
\]
The following result characterizes the breakdown point of the empirical transport problem \eqref{DiscreteOT}.
\begin{theorem}    \label{Theo:mainDiscrete2}
    Let $X^{(n)}=\{X_1, \dots, X_n\}$, $u^{(n)}=\{u_1, \dots, u_n\}$ and $\hat{T}$ be as in \eqref{DiscreteOT}. Then the breakdown point of $\hat{T}(u_i)$ lies in between  the lower Tukey depth and the Tukey depth of $u_i$ with respect to the dataset $u^{(n)}$, i.e.,   
    $$
    {\rm BP}(\hat{T}(u_i))\in [\mathrm{TD}^{-}(u_i; u^{(n)}),\mathrm{TD}(u_i; u^{(n)})].
$$
\end{theorem}
\begin{remark}\label{Rem:sharpness}
    As illustrated in \autoref{fig:Sharp} the lower bound in \autoref{Theo:mainDiscrete} is attained for some reference measures. The upper bound is also attained when we consider the reference and target datasets both equal to  $\{1,2,3,4,5\}\subset \R$. The Tukey depth of $3$ is $3/5$ and the breakdown point of the optimal transport map evaluated at $ 3$ is $3/5$. As a consequence, the bounds of \autoref{Theo:mainDiscrete}  are sharp. 
\end{remark}
\begin{remark}\label{Remark:ArbitralyBad}
    We also observe that  $ \mathrm{TD}^{-}(u_i;u^{(n)})$ and $\mathrm{TD}(u_i; u^{(n)})$ get closer as $n\to \infty$ if $\mu_n=\frac{1}{n}\sum_{i=1}^n\delta_{u_i}$ converges weakly to a probability measure that is absolutely continuous  with respect to  the Lebesgue measure. We underline that the breakdown point of the optimal transport map could be arbitrary small. Indeed,  if the reference sample $u^{(n)}$ is supported on a lower dimensional affine subspace and $X^{(n)}$ is very close to $u^{(n)}$,   a perturbation of  {a} single point will send  the optimal transport map evaluated at a particular $ u_i\in u^{(n)} $ to infinity  (see the construction of \autoref{fig:Sharp}). Thus, as the reference sample becomes more disordered or random, the robustness of the optimal transport map increases.
\end{remark}
\begin{figure}[h!]
    \centering
    \resizebox{0.99\textwidth}{!}{ 

\tikzset{every picture/.style={line width=0.75pt}} 

\begin{tikzpicture}[x=0.75pt,y=0.75pt,yscale=-1,xscale=1]

\draw [color={rgb, 255:red, 74; green, 144; blue, 226 }  ,draw opacity=1 ][line width=0.75]  [dash pattern={on 0.84pt off 2.51pt}]  (159,26) -- (159.17,60.83) -- (159.33,95.67) -- (159.5,130.5) -- (159.67,165.33) -- (159.83,200.17) -- (160,235) ;
\draw [shift={(160,235)}, rotate = 89.73] [color={rgb, 255:red, 74; green, 144; blue, 226 }  ,draw opacity=1 ][fill={rgb, 255:red, 74; green, 144; blue, 226 }  ,fill opacity=1 ][line width=0.75]      (0, 0) circle [x radius= 3.35, y radius= 3.35]   ;
\draw [shift={(159,26)}, rotate = 89.73] [color={rgb, 255:red, 74; green, 144; blue, 226 }  ,draw opacity=1 ][fill={rgb, 255:red, 74; green, 144; blue, 226 }  ,fill opacity=1 ][line width=0.75]      (0, 0) circle [x radius= 3.35, y radius= 3.35]   ;
\draw [color={rgb, 255:red, 74; green, 144; blue, 226 }  ,draw opacity=1 ][line width=0.75]  [dash pattern={on 0.84pt off 2.51pt}]  (159.5,130.5) -- (273.5,130.5) ;
\draw [shift={(273.5,130.5)}, rotate = 0] [color={rgb, 255:red, 74; green, 144; blue, 226 }  ,draw opacity=1 ][fill={rgb, 255:red, 74; green, 144; blue, 226 }  ,fill opacity=1 ][line width=0.75]      (0, 0) circle [x radius= 3.35, y radius= 3.35]   ;
\draw [shift={(159.5,130.5)}, rotate = 0] [color={rgb, 255:red, 74; green, 144; blue, 226 }  ,draw opacity=1 ][fill={rgb, 255:red, 74; green, 144; blue, 226 }  ,fill opacity=1 ][line width=0.75]      (0, 0) circle [x radius= 3.35, y radius= 3.35]   ;
\draw [color={rgb, 255:red, 74; green, 144; blue, 226 }  ,draw opacity=1 ][line width=0.75]  [dash pattern={on 0.84pt off 2.51pt}]  (45.5,130.5) -- (159.5,130.5) ;
\draw [shift={(159.5,130.5)}, rotate = 0] [color={rgb, 255:red, 74; green, 144; blue, 226 }  ,draw opacity=1 ][fill={rgb, 255:red, 74; green, 144; blue, 226 }  ,fill opacity=1 ][line width=0.75]      (0, 0) circle [x radius= 3.35, y radius= 3.35]   ;
\draw [shift={(45.5,130.5)}, rotate = 0] [color={rgb, 255:red, 74; green, 144; blue, 226 }  ,draw opacity=1 ][fill={rgb, 255:red, 74; green, 144; blue, 226 }  ,fill opacity=1 ][line width=0.75]      (0, 0) circle [x radius= 3.35, y radius= 3.35]   ;
\draw [color={rgb, 255:red, 208; green, 2; blue, 27 }  ,draw opacity=1 ][line width=0.75]  [dash pattern={on 0.84pt off 2.51pt}]  (494,25) -- (494.17,59.83) -- (494.33,94.67) -- (494.5,129.5) -- (494.67,164.33) -- (494.83,199.17) -- (495,234) ;
\draw [shift={(495,234)}, rotate = 89.73] [color={rgb, 255:red, 208; green, 2; blue, 27 }  ,draw opacity=1 ][fill={rgb, 255:red, 208; green, 2; blue, 27 }  ,fill opacity=1 ][line width=0.75]      (0, 0) circle [x radius= 3.35, y radius= 3.35]   ;
\draw [shift={(494,25)}, rotate = 89.73] [color={rgb, 255:red, 208; green, 2; blue, 27 }  ,draw opacity=1 ][fill={rgb, 255:red, 208; green, 2; blue, 27 }  ,fill opacity=1 ][line width=0.75]      (0, 0) circle [x radius= 3.35, y radius= 3.35]   ;
\draw [color={rgb, 255:red, 208; green, 2; blue, 27 }  ,draw opacity=1 ][line width=0.75]  [dash pattern={on 0.84pt off 2.51pt}]  (629,130) ;
\draw [shift={(629,130)}, rotate = 0] [color={rgb, 255:red, 208; green, 2; blue, 27 }  ,draw opacity=1 ][fill={rgb, 255:red, 208; green, 2; blue, 27 }  ,fill opacity=1 ][line width=0.75]      (0, 0) circle [x radius= 3.35, y radius= 3.35]   ;
\draw [shift={(629,130)}, rotate = 0] [color={rgb, 255:red, 208; green, 2; blue, 27 }  ,draw opacity=1 ][fill={rgb, 255:red, 208; green, 2; blue, 27 }  ,fill opacity=1 ][line width=0.75]      (0, 0) circle [x radius= 3.35, y radius= 3.35]   ;
\draw [color={rgb, 255:red, 208; green, 2; blue, 27 }  ,draw opacity=1 ][line width=0.75]  [dash pattern={on 0.84pt off 2.51pt}]  (380.5,129.5) -- (742,130) ;
\draw [shift={(380.5,129.5)}, rotate = 0.08] [color={rgb, 255:red, 208; green, 2; blue, 27 }  ,draw opacity=1 ][fill={rgb, 255:red, 208; green, 2; blue, 27 }  ,fill opacity=1 ][line width=0.75]      (0, 0) circle [x radius= 3.35, y radius= 3.35]   ;
\draw [color={rgb, 255:red, 208; green, 2; blue, 27 }  ,draw opacity=1 ][line width=0.75]  [dash pattern={on 0.84pt off 2.51pt}]  (742,130) ;
\draw [shift={(742,130)}, rotate = 0] [color={rgb, 255:red, 208; green, 2; blue, 27 }  ,draw opacity=1 ][fill={rgb, 255:red, 208; green, 2; blue, 27 }  ,fill opacity=1 ][line width=0.75]      (0, 0) circle [x radius= 3.35, y radius= 3.35]   ;
\draw [shift={(742,130)}, rotate = 0] [color={rgb, 255:red, 208; green, 2; blue, 27 }  ,draw opacity=1 ][fill={rgb, 255:red, 208; green, 2; blue, 27 }  ,fill opacity=1 ][line width=0.75]      (0, 0) circle [x radius= 3.35, y radius= 3.35]   ;
\draw    (159,26) -- (491,25.01) ;
\draw [shift={(494,25)}, rotate = 179.83] [fill={rgb, 255:red, 0; green, 0; blue, 0 }  ][line width=0.08]  [draw opacity=0] (10.72,-5.15) -- (0,0) -- (10.72,5.15) -- (7.12,0) -- cycle    ;
\draw    (45.5,130.5) .. controls (65.69,87.65) and (348.33,78.28) .. (379.26,127.22) ;
\draw [shift={(380.5,129.5)}, rotate = 245.47] [fill={rgb, 255:red, 0; green, 0; blue, 0 }  ][line width=0.08]  [draw opacity=0] (10.72,-5.15) -- (0,0) -- (10.72,5.15) -- (7.12,0) -- cycle    ;
\draw [color={rgb, 255:red, 0; green, 0; blue, 0 }  ,draw opacity=1 ]   (273.5,130.5) .. controls (276.96,85.45) and (709.71,78.63) .. (741.17,128.47) ;
\draw [shift={(742,130)}, rotate = 245.47] [color={rgb, 255:red, 0; green, 0; blue, 0 }  ,draw opacity=1 ][line width=0.75]    (10.93,-3.29) .. controls (6.95,-1.4) and (3.31,-0.3) .. (0,0) .. controls (3.31,0.3) and (6.95,1.4) .. (10.93,3.29)   ;
\draw    (159.5,130.5) .. controls (198.8,131.99) and (539.57,218.13) .. (627.69,131.32) ;
\draw [shift={(629,130)}, rotate = 134.02] [fill={rgb, 255:red, 0; green, 0; blue, 0 }  ][line width=0.08]  [draw opacity=0] (10.72,-5.15) -- (0,0) -- (10.72,5.15) -- (7.12,0) -- cycle    ;
\draw    (160,235) -- (492,234.01) ;
\draw [shift={(495,234)}, rotate = 179.83] [fill={rgb, 255:red, 0; green, 0; blue, 0 }  ][line width=0.08]  [draw opacity=0] (10.72,-5.15) -- (0,0) -- (10.72,5.15) -- (7.12,0) -- cycle    ;

\draw (159.5,139.9) node [anchor=north west][inner sep=0.75pt]    {$( 0,0)$};
\draw (275.5,133.9) node [anchor=north west][inner sep=0.75pt]    {$( 1,0)$};
\draw (47.5,133.9) node [anchor=north west][inner sep=0.75pt]    {$( -1,0)$};
\draw (161,29.4) node [anchor=north west][inner sep=0.75pt]    {$( 0,1)$};
\draw (162,238.4) node [anchor=north west][inner sep=0.75pt]    {$( -1,0)$};
\draw (631,133.4) node [anchor=north west][inner sep=0.75pt]    {$( k,0)$};
\draw (744,133.4) node [anchor=north west][inner sep=0.75pt]    {$( k+1,0)$};
\draw (382.5,132.9) node [anchor=north west][inner sep=0.75pt]    {$( -1,0)$};
\draw (496,28.4) node [anchor=north west][inner sep=0.75pt]    {$( 0,1)$};
\draw (497,237.4) node [anchor=north west][inner sep=0.75pt]    {$( -1,0)$};

\end{tikzpicture}

}
    \caption{Example showing that the lower bound of \autoref{Theo:mainDiscrete} is sharp. The reference and target datasets $u^{(n)}$ and $X^{(n)}$ are both $\{(0,0), (1,0), (0,1), (-1, 0), (0,-1)\}$. The Tukey depth of the origin is $ 3/5$ and the lower Tukey  depth is $2/5$. Perturbing two points of $X^{(n)}$  as in the red dataset of the picture, we find that the optimal transport map from the blue the the  red points maps $(0,0)$ to $(0,k)$, which escapes to the horizon. Therefore the breakdown point of the optimal transport map evaluated at the origin is $2/5$.}
    \label{fig:Sharp}
\end{figure}

We now present our main theorem, which asserts that the breakdown point of $\partial \varphi = \partial \varphi_{\mu \to P}$, evaluated at $u \in \mathbb{R}^d$, corresponds to the Tukey depth of $u$ in the reference measure $\mu$. 
Recall that $\partial \varphi$ is the subdifferential of a l.s.c.\ convex function $\varphi$ whose gradient pushes $\mu$ forward to $P$. 
\begin{theorem}    \label{Theo:mainCont}
    Set $\mu,P\in \mathcal{P}^{a.c.}(\R^d)$. Let $\nabla \varphi $ be the optimal transport map  pushing $\mu$ forward to $P$.   Then the breakdown point of a canonical extension $\partial \varphi$ at a point $u$ with $\mathrm{TD}(u; \mu)>0$  is the Tukey depth of  $u$  with respect to  $\mu$, i.e.,   
    $${\rm BP}(\partial \varphi(u), P)=\mathrm{TD}(u; \mu), \quad \text{irrespective of $\mu,P \in \mathcal{P}^{a.c.}(\R^d)$.}
$$
\end{theorem}

\begin{remark}\label{rem:anotherBP}After inspecting the proof of \autoref{Theo:mainCont} (cf. \autoref{lemma:LowerCont} and \autoref{lemma:LowerCont2}), we see that a stronger notion of breakdown point leads to the same quantitative result. Namely, for every $u\in \R^d$ such that $\mathrm{TD}(u; \mu)>0$, 
    $$\overline{\rm BP}(\partial \varphi(u), P)= \min \left\{ \eps\in (0,1): \sup_{\varphi_{\mu\to \nu}\in \Gamma_\eps(P)} \sup_{v\in \partial \varphi_{\mu\to \nu} (u)} \|v\| = \infty\right \}=\mathrm{TD}(u; \mu).$$
 \end{remark}
 {
\begin{remark}[General cost functions]
Our proofs are based on \autoref{lemma:conetocone}, where monotonicity is the main ingredient, and thus they are not easily adaptable to general transport costs. We conjecture that for costs of the form $\|x-y\|^p$—or, more generally, for the costs considered in \cite{GangboMcCann.96}—our results remain valid.
\end{remark}
}

\subsection{Breakdown points of the transport-based quantiles}

We can now easily obtain the breakdown point of the empirical and population depth contours using the results from the previous subsection. \autoref{Theo:mainDiscrete2}
 yields the breakdown points of the transport-based  median and depth contours for discrete probability measures, while \autoref{Theo:mainCont} gives the result for their continuous counterparts. In other words, we have obtained the finite sample and asymptotic breakdown points of the transport-based medians and depth contours.
\begin{corollary}\label{coro:median_contours_disc}
Let x Let $X^{(n)}=\{X_1, \dots, X_n\}$ and  $u^{(n)}=\{u_1, \dots, u_n\}$ be respectively the observed and reference samples. Fix  {$$N=\max_{(v,u)\in \mathcal{S}^{d-1}\times \R^d}\left\{\sum_{i=1}^{n} \mathbf{1}_{\left( \langle v, u_i - u \rangle = 0 \right)} \right\}.$$}  { Then, for every $\tau\in [0,1/2]$ with  ${\cal  C}_\pm^{(n)}(\tau)\neq \emptyset$, 
 $$ {\rm BP}({\cal  C}_\pm^{(n)}(\tau) )\in [\tau-N/n,\tau], \quad  \text{irrespective of the dataset $X_1, \dots, X_n$.  } $$ }
 \end{corollary}
 \begin{remark}
      {
The case of the median is a particular instance of \autoref{coro:median_contours_disc}. In particular, we have 
  $$ {\rm BP}(m_\pm^{(n)})\in \left[\tau^*-N/n,\tau^*\right], \quad \text{for } 
\tau^*:=\max_{u\in \R^d} {\rm TD}( u; u^{(n)}).  $$ 
}
 \end{remark}

\begin{corollary}\label{coro:median_contours_cont}
Take a reference measure $\mu\in \mathcal{P}^{a.c.}(\R^d)$ such that $\max_{x\in \R^d} {\rm TD}( x; \mu) =1/2$. Then, for any  $P\in \mathcal{P}^{a.c.}(\R^d)$, 
    $${\rm BP}({m
}_\pm({P}), P)=\frac{1}{2} $$ 
and
 $$ {\rm BP}({\cal  C}_\pm(\tau) ) =\tau \quad \text{holds for all $\tau\in (0, 1/2 )$} $$
\end{corollary}

\begin{remark}
    Note that \autoref{coro:median_contours_disc} does not follow directly from \autoref{Theo:mainDiscrete2}.
     This is  because we defined $m_\pm^{(n)}$ and ${\cal  C}_\pm^{(n)}(\tau) $ using maximally monotone extensions (beyond the sample points) while \autoref{Theo:mainDiscrete2} focuses only on the behavior at the sample points.  { In fact, we will prove both results simultaneously in \autoref{Sect:proofDisc}.} 
\end{remark}
\begin{remark}\label{rem:reference_measure2}From \autoref{coro:median_contours_disc} and \autoref{coro:median_contours_cont} we conclude that the reference measure does not have any influence on the population breakdown point  as long as it attains the maximal Tukey depth of $1/2$. However, in the discrete case it might have an small impact as illustrated in  \autoref{Remark:ArbitralyBad}.
\end{remark}

\section{Proof of main results}\label{Sect:proofs}
This section is devoted to the proofs of \autoref{Theo:mainDiscrete2} and \autoref{Theo:mainCont}. Both arguments follow the same idea, which is based on separating a portion of the data into a ball escaping to the horizon  in the direction given by $v_{u}\in\argmin_{v\in \mathcal{S}^{d-1}} \mu\left( \{z: \langle v, z - u \rangle \geq 0\} \right) $. We will argue that a claimed breakdown point $\varepsilon^*$ holds by contradiction. We will show that if $\varepsilon^*$ is not the breakdown, then the transport-based quantile function (or its empirical counterpart) would have to transform a convex cone into a non-convex cone that violates the mass-preserving property of transport maps. More precisely, we show that  any monotone set-valued mapping $T:\R^d\to 2^{\R^d}$  transforms the convex cone
\begin{equation}
    \label{eq:convexcone}
    \mathcal{C}_{x,e, \theta} =\{ x+v\in \R^d: \ \langle v, e\rangle \geq \cos(\theta) \|v\|\} \quad  {{\rm for}\  x\in \R^d, \ e\in \mathcal{S}^{d-1}\ {\rm and}\  \theta\in \left[-\frac{\pi}{2}, \frac{\pi}{2}\right],} 
\end{equation}
 {into the cone $ \mathcal{A}_{y^x,e, \theta} $ for any $y^x\in T(x)$, where }
\begin{equation}
    \label{eq:cone-def-nonconvex}
    \mathcal{A}_{u,e, \theta} =\left\{y\in \R^d: \langle u-y, e\rangle\leq 
 \| u-y\|  {|\sin(\theta)|}   \right\}. 
\end{equation}
We derive a contradiction by showing that if $T$ pushes forward  $\mu$
to $P$ (or its empirical measure) and $\varepsilon^*$ is not the breakdown, then the $P$-mass of  $\mathcal{C}_{x,e, \theta}$ is larger than the $\mu$-mass of $ \mathcal{A}_{u,e, \theta}$.  {We will show that $T(\mathcal{C}_{x,e, \theta})\subset \mathcal{A}_{u,e, \theta}$, which will then contradict the mass-preserving property of the push forward $T$.}
\begin{figure}[ht!]
    \centering
    \resizebox{0.99\textwidth}{!}{ 

    \tikzset{every picture/.style={line width=0.75pt}} 

\tikzset{every picture/.style={line width=0.75pt}} 

\tikzset{every picture/.style={line width=0.75pt}} 

\begin{tikzpicture}[x=0.75pt,y=0.75pt,yscale=-1,xscale=1]

\draw   (293.8,131.2) .. controls (293.8,120.15) and (302.98,111.2) .. (314.3,111.2) .. controls (325.62,111.2) and (334.8,120.15) .. (334.8,131.2) .. controls (334.8,142.25) and (325.62,151.2) .. (314.3,151.2) .. controls (302.98,151.2) and (293.8,142.25) .. (293.8,131.2) -- cycle ;
\draw   (339.6,128.45) -- (357,128.45) -- (357,124.2) -- (368.6,132.7) -- (357,141.2) -- (357,136.95) -- (339.6,136.95) -- cycle ;
\draw [fill={rgb, 255:red, 194; green, 28; blue, 28 }  ,fill opacity=1 ][line width=1.5]  [dash pattern={on 1.69pt off 2.76pt}]  (147.8,139.2) -- (370.8,100.2) ;
\draw [line width=1.5]  [dash pattern={on 1.69pt off 2.76pt}]  (147.8,139.2) -- (370.8,156.2) ;
\draw   (34.8,139.2) .. controls (34.8,108) and (60.1,82.7) .. (91.3,82.7) .. controls (122.5,82.7) and (147.8,108) .. (147.8,139.2) .. controls (147.8,170.4) and (122.5,195.7) .. (91.3,195.7) .. controls (60.1,195.7) and (34.8,170.4) .. (34.8,139.2) -- cycle ;
\draw    (91.3,139.2) -- (39.76,149.8) ;
\draw [shift={(37.8,150.2)}, rotate = 348.38] [color={rgb, 255:red, 0; green, 0; blue, 0 }  ][line width=0.75]    (10.93,-3.29) .. controls (6.95,-1.4) and (3.31,-0.3) .. (0,0) .. controls (3.31,0.3) and (6.95,1.4) .. (10.93,3.29)   ;
\draw  [draw opacity=0][dash pattern={on 1.69pt off 2.76pt}][line width=1.5]  (226.19,144.2) .. controls (226.2,143.76) and (226.2,143.33) .. (226.2,142.9) .. controls (226.2,137.11) and (225.73,131.43) .. (224.84,125.91) -- (136.8,142.9) -- cycle ; \draw  [dash pattern={on 1.69pt off 2.76pt}][line width=1.5]  (226.19,144.2) .. controls (226.2,143.76) and (226.2,143.33) .. (226.2,142.9) .. controls (226.2,137.11) and (225.73,131.43) .. (224.84,125.91) ;  
\draw   (420.6,128) .. controls (420.6,74.54) and (463.94,31.2) .. (517.4,31.2) .. controls (570.86,31.2) and (614.2,74.54) .. (614.2,128) .. controls (614.2,181.46) and (570.86,224.8) .. (517.4,224.8) .. controls (463.94,224.8) and (420.6,181.46) .. (420.6,128) -- cycle ;
\draw [color={rgb, 255:red, 0; green, 0; blue, 0 }  ,draw opacity=1 ]   (147.8,139.2) .. controls (150.18,80.79) and (451.29,106.08) .. (517.47,130.85) ;
\draw [shift={(519.4,131.6)}, rotate = 201.96] [fill={rgb, 255:red, 0; green, 0; blue, 0 }  ,fill opacity=1 ][line width=0.08]  [draw opacity=0] (8.93,-4.29) -- (0,0) -- (8.93,4.29) -- cycle    ;
\draw [fill={rgb, 255:red, 194; green, 28; blue, 28 }  ,fill opacity=1 ][line width=1.5]  [dash pattern={on 1.69pt off 2.76pt}]  (479.8,239.4) -- (519.4,131.6) ;
\draw [fill={rgb, 255:red, 194; green, 28; blue, 28 }  ,fill opacity=1 ][line width=1.5]  [dash pattern={on 1.69pt off 2.76pt}]  (519.4,131.6) -- (477.4,18.6) ;
\draw   (494.8,211.2) -- (509.92,211.2) -- (509.92,209) -- (520,213.4) -- (509.92,217.8) -- (509.92,215.6) -- (494.8,215.6) -- cycle ;
\draw   (492.4,46.05) -- (507.16,46.05) -- (507.16,44) -- (517,48.1) -- (507.16,52.2) -- (507.16,50.15) -- (492.4,50.15) -- cycle ;
\draw [color={rgb, 255:red, 0; green, 0; blue, 0 }  ,draw opacity=1 ]   (312.6,141.8) .. controls (333.09,199.32) and (538.79,227.74) .. (567.51,178.31) ;
\draw [shift={(568.7,176)}, rotate = 114.16] [fill={rgb, 255:red, 0; green, 0; blue, 0 }  ,fill opacity=1 ][line width=0.08]  [draw opacity=0] (8.93,-4.29) -- (0,0) -- (8.93,4.29) -- cycle    ;
\draw [line width=1.5]  [dash pattern={on 1.69pt off 2.76pt}]  (519,14.6) -- (523,245) ;
\draw  [draw opacity=0] (500.04,85.16) .. controls (505.32,82.49) and (511.2,81) .. (517.4,81) .. controls (517.66,81) and (517.92,81) .. (518.18,81.01) -- (517.4,125.85) -- cycle ; \draw   (500.04,85.16) .. controls (505.32,82.49) and (511.2,81) .. (517.4,81) .. controls (517.66,81) and (517.92,81) .. (518.18,81.01) ;  
\draw  [draw opacity=0] (521.03,175.81) .. controls (515.15,176.41) and (509.12,175.69) .. (503.33,173.47) .. controls (503.09,173.38) and (502.85,173.28) .. (502.61,173.19) -- (519.4,131.6) -- cycle ; \draw   (521.03,175.81) .. controls (515.15,176.41) and (509.12,175.69) .. (503.33,173.47) .. controls (503.09,173.38) and (502.85,173.28) .. (502.61,173.19) ;  

\draw (301.8,123.2) node [anchor=north west][inner sep=0.75pt]   [align=left] {$ $$\displaystyle Q_{k}$$ $};
\draw (126.8,127.6) node [anchor=north west][inner sep=0.75pt]    {$x_{k}$};
\draw (231.6,126.8) node [anchor=north west][inner sep=0.75pt]    {$\theta _{k}$};
\draw (59.6,150) node [anchor=north west][inner sep=0.75pt]    {$R$};
\draw (86.8,122) node [anchor=north west][inner sep=0.75pt]    {$0$};
\draw (524.4,122) node [anchor=north west][inner sep=0.75pt]    {$u$};
\draw (224.4,82.8) node [anchor=north west][inner sep=0.75pt]    {$(\partial \varphi_{\eps,k})^{-1}$};
\draw (501.2,61.2) node [anchor=north west][inner sep=0.75pt]    {$\theta _{k}$};
\draw (502.8,181.2) node [anchor=north west][inner sep=0.75pt]    {$\theta _{k}$};
\draw (328.4,190.8) node [anchor=north west][inner sep=0.75pt]    {$(\partial \varphi_{\eps,k})^{-1}$};
\draw (502,4.4) node [anchor=north west][inner sep=0.75pt]    {$H$};
\draw (461,11.4) node [anchor=north west][inner sep=0.75pt]    {$r_{1}$};
\draw (463,222.4) node [anchor=north west][inner sep=0.75pt]    {$r_{2}$};

\end{tikzpicture}}
    \caption{Diagram illustrating the strategy for the proof of \autoref{Theo:mainCont}.
    For the upper bound, we consider the mixture \((1-\eps)P + \eps Q_k\), where \(Q_k\) is uniformly distributed on the ball $k v_u+\mathbb{B}$ for  $v_u\in\argmin_{v\in \mathcal{S}^{d-1}} \mu\left( \{z: \langle v, z - u \rangle \geq 0\} \right) $. If \(\eps > {\rm TD}(u;\mu)\) and a sequence \(\{x_k\}_{k \in \mathbb{N}}\)  {with} $x_k\in \partial \varphi_{\eps,k}(u):=\partial \varphi_{\mu\to ( 1-\eps ) P+\eps Q_{k}}(u)$ remains bounded within a ball of radius \(R\), then for sufficiently large \(k\), the map \((\partial \varphi_{\eps,k})^{-1}\) transforms the cone with vertex \(x_k\) and angle \(\theta_k\) into a set located to the right of the rays \(r_1\) and \(r_2\). As the ball escapes to the horizon, the angle \(\theta_k\) tends to \(0\), causing the image of the ball to be contained within the left halfspace defined by the hyperplane \(H\), which is tangent to $v_u$. Consequently, as \(\eps > {\rm TD}(u;\mu)\), the \(((1-\eps)P+\eps Q_k)\)-measure of the cone exceeds \({\rm TD}(u;\mu)\), while its image under \((\partial \varphi_{\eps,k})^{-1}\) converges to a set with \(\mu\)-probability of \({\rm TD}(u;\mu)\). This leads to a contradiction. 
    The strategy for the lower bound follows a similar approach.}
    \label{fig:strategy}
\end{figure}

The proof idea of the contradiction is illustrated in \autoref{fig:strategy}.   \autoref{lemma:conetocone} formalizes that a monotone set-valued map $T$ transforms $\mathcal{C}_{x,e, \theta}$ into $ \mathcal{A}_{u,e, \theta}$. Interested readers may note that similar results are commonly employed to show that a transport map between probability measures $\mu_1$ and $\mu_2$, where $\mu_2$
  is supported on a strongly convex set, cannot map interior points of  $\supp(\mu_1)$ to boundary points of $\supp(\mu_2)$ (cf. \cite{CorderoErausquinFigalli.2019,delBarrioGonzalezHallin.2020}). 
\begin{lemma}\label{lemma:conetocone}
     Let $T:{\rm dom}(T)\subset \R^d\to 2^{\R^d}$ be a monotone operator, i.e., 
     \begin{equation}
         \label{monoton}
         \langle x_1-x_2, y_1-y_2\rangle\geq 0, \quad \text{for all } x_1,x_2\in {\rm dom}(T),\ y_1\in T(x_1) \ \text{and}\  y_2\in T(x_2) .
     \end{equation}
    Set $ x\in  {\rm dom}(T) $ and  {recall the definitions \eqref{eq:convexcone} and \eqref{eq:cone-def-nonconvex}.} Then, for any  $\theta\in (0, \pi/2)$ and any $e\in \mathcal{S}^{d-1}$, it holds that $$T(\mathcal{C}_{x,e, \theta}):=\bigcup_{z\in\mathcal{C}_{x,e, \theta} } T(z)\subset \mathcal{A}_{y^x,e, \theta},$$
    for all $y^x\in T(x)$. 
 \end{lemma}
 \begin{proof}
Firstly, we construct an  orthonormal basis $\{e_i\}_{i=1}^d$ with $e_1=e$. For a vector $\mathtt{v}\in\mathbb{R}^d$ we denote as $(\mathtt{v}_1, \dots, \mathtt{v}_d)$ its coordinates  with respect to  the basis $\{e_i\}_{i=1}^d$.     Fix $z\in \mathcal{C}_{x,e_1, \theta} $. By definition of the cone,   $z_1-x_1\geq \cos(\theta) \|z-x\|$.      {Hence, by Pythagoras' theorem 
$$\|x-z\|^2=\|(x-z)-(x_1-z_1)e_1\|^2+ (x_1-z_1)^2 \geq \|(x-z)-(x_1-z_1)e_1\|^2+ \cos(\theta)^2 \|z-x\|^2.  $$
Therefore, we derive the bound  
\begin{align}\label{eq:boundSeno}
   \|(x-z)-(x_1-z_1)e_1\|\leq  {|\sin(\theta)|}  (z_1-x_1).
\end{align}}
Fix $y^z\in T(z)$ and $y^x\in T(x)$. Then, by \eqref{monoton}, it holds that $0\leq \langle y^x-y^z, x-z\rangle$,   so that 
$$ 0\leq \langle y^x-y^z, x-z\rangle =  \langle y^x-y^z, x-z-(x_1-z_1)e_1\rangle+(x_1-z_1)\langle y^x-y^z, e_1\rangle.$$
Rearranging the terms and using  Cauchy–Schwarz inequality we get 
\begin{align*}
    (z_1-x_1)\langle y^x-y^z, e_1\rangle&\leq   \langle y^x-y^z, x-z-(x_1-z_1)e_1\rangle\\
    &\leq \| y^x-y^z\| \|x-z-(x_1-z_1)e_1\|.
\end{align*}
By \eqref{eq:boundSeno}, it holds that 
\begin{equation*}
    (z_1-x_1)\langle y^x-y^z, e_1\rangle\leq \| y^x-y^z\|  {|\sin(\theta)|}  (z_1-x_1) {.}
\end{equation*}
Since $z_1-x_1>0$,  {we can divide both sides by $z_1-x_1$} and the result follows.
 \end{proof}

\subsection{Population transport map: proof of \autoref{Theo:mainCont}}\label{Sect:proofCont}
We start with a straightforward observation, which follows directly from  {  (Fatou’s version of) the continuity of the probability measure (see e.g.,~\cite[Theorem 4.1]{Billingsley.1995}).}
\begin{lemma}\label{lemma:ConvergenceToOneHalph}
  Set $\mu\in \mathcal{P}(\R^d)$.  Let $\{\theta_k\}_{k\in \N}\subset (0,\pi/2)$ and $\{v_k\}_{k\in \N}\subset \mathbb{R}^d$ be such that  $ \theta_k\to 0$ and $ v_k\to v$, respectively, as $k\to \infty$. Then
$$ \limsup_{k\to \infty} \mu(\mathcal{A}_{u,v_k, \theta_k}) \leq \mu(\{ y: 
      \langle  y-u, v\rangle\geq 
0  \} ), $$
 {where $\mathcal{A}_{u,v_k, \theta_k}$ is defined in \eqref{eq:cone-def-nonconvex}.}
\end{lemma}

We will prove that ${\rm BP}(\partial \varphi(u), P)=\mathrm{TD}(u; \mu)$ by showing matching  upper and lower bounds separately.  Recall that for a probability measure $\nu$, the Borel map $\nabla  \varphi_{\mu\to \nu}$ is the unique gradient of a convex function pushing $\mu$ forward to $\nu$.  For the case $\nu=P$, we just write $\nabla  \varphi=\nabla  \varphi_{\mu\to P}$.
\autoref{Lemma:TukeyPositiveCont} ensures that  a point  with (strictly) positive Tukey depth has always non-empty subdifferential.   The following lemma shows the upper bound. 
\begin{lemma}\label{lemma:LowerCont}
Let $\mu,P\in \mathcal{P}^{a.c.}(\R^d)$ and $Q_k$ be the Lebesgue uniform distribution on the ball $k v_u+\mathbb{B}$, $k\in \N$, for some 
 $$v_u\in \argmin_{v\in \mathcal{S}^{d-1}} \mu(\{ y: 
      \langle  y-u, v\rangle\geq 
0  \} ).$$
Then, for every $\eps>{\rm TD}(u;\mu)>0$, the sequence $\{\partial \varphi_{\eps,k}(u)\}_{k\in \N}$, with  
 $\partial \varphi_{\eps,k} $  being any maximal {  cyclical} monotone extension of $\nabla \varphi_{\mu\to P+\eps (Q_k-P)}$, escapes to the horizon as $k\to +\infty$, i.e.,   
    $$\inf_{y\in \partial \varphi_{\eps,k}(u)}\|y\|\to +\infty.$$
\end{lemma}
\begin{proof}
Note first that as ${\rm TD}(u;\mu)>0$, \autoref{Lemma:TukeyPositiveCont} implies that $ \partial\varphi_{\eps,k}(u) \neq \emptyset$ for all $k\in \N$ and  any maximal {  cyclical} monotone extension of $\nabla \varphi_{\mu\to P+\eps (Q_k-P)}$. 
We will show that for any sequence $
\{x_k\}_{k\in \N}$, with 
$ x_k \in \partial\varphi_{\eps,k}(u)$
for all $k\in \N$, it must hold that $ \|x_k\| \to \infty.$ 
To prove the claim we argue by contradiction. To do so, we can assume, after taking subsequences, that $x_k$ belongs to some ball $R\,\mathbb{B}$ for all $k\in \N$ and some finite $R>0$. 
Assume that $k\geq  2 (R+1) $ and let 
$ \mathcal{D}_k $ be the set of all rays arising from $ x_k $ and  {intersecting} $k v_u+\mathbb{B}$. Note that for all $k$ large enough there exists $e_k\in \mathcal{S}^{d-1}$ and $\theta_k\in (0,\pi/2)$ such that  $ {{\rm cl}}(\mathcal{D}_k) =\mathcal{C}_{x_k,e_k,\theta_k}$ as defined  in \eqref{eq:convexcone}.

On the one hand, it holds that  
\begin{equation}
    \label{eq:ToContradictlowerbound1}
    (P+\eps (Q_{k}-P))(\mathcal{C}_{x_k,e_k,\theta_k})\geq  \eps Q_{k}(\mathcal{C}_{x_k,e_k,\theta_k})= \eps  {\geq } {\rm TD}(u;\mu)+\alpha
\end{equation}
for some $\alpha>0$. On the other hand, \autoref{lemma:conetocone}  and \eqref{MonotoneQuant} imply that 
$$ (\partial \varphi_{\eps,k})^{-1}(\mathcal{C}_{x_k,e_k,\theta_k})\subset \mathcal{A}_{u,e_k,\theta_k}, $$
which yields
\begin{equation}
    \label{eq:ToContradictlowerbound2}
    \mu((\partial \varphi_{\eps,k})^{-1}(\mathcal{C}_{x_k,e_k,\theta_k})) {\leq } \mu(\mathcal{A}_{u,e_k,\theta_k}).
\end{equation}

Note that $\theta_k\to 0  $ and $e_k\to v_u$, due to $\|x_k\|\leq R$ and the fact that $k v_u+\mathbb{B}$ escapes to the horizon.   The contradiction will be found by showing that as $k\to \infty$, the optimal transport map $\nabla \varphi_{\eps,k}^*=(\nabla \varphi_{\eps,k})^{-1}$ violates the push forward condition. 
Since  $\partial  \varphi_{\eps,k}^*= (\partial  \varphi_{\eps,k})^{-1}= \{\nabla  \varphi_{\eps,k}^*(x)\}$ for $(P+\eps (Q_{k}-P))$-a.e.\ $x$ and the Borel map $\nabla  \varphi_{\eps,k}^*$ pushes $P+\eps (Q_{k}-P)$ forward to $\mu$, it must hold that 
$$ \mu(\partial  \varphi_{\eps,k}^*(\mathcal{C}_{x_k,e_k,\theta_n}))= (P+\eps (Q_{k}-P))(\mathcal{C}_{x_k,e_k,\theta_k}).$$
We use \eqref{eq:ToContradictlowerbound1} and \eqref{eq:ToContradictlowerbound2} to get 
$ \mu(\mathcal{A}_{u,e_k,\theta_k}) {\geq } {\rm TD}(u;\mu)+\alpha.$ Since $\theta_k\to 0$ and $e_k\to v_u$, \autoref{lemma:ConvergenceToOneHalph} leads to the contradiction
\begin{align*}
      {\rm TD}(u;\mu)&=\min_{v\in \mathcal{S}^{d-1}}\mu(\{ y: 
      \langle  y-u, v\rangle\geq 
0  \} )\\
&=\mu(\{ y: 
      \langle  y-u, v_u\rangle\geq 
0  \} )
\geq  {\limsup_{k\to \infty}}\mu(\mathcal{A}_{u,e_k,\theta_k}) {\geq }{\rm TD}(u;\mu)+\alpha.
\end{align*}
Therefore, the result follows. 
\end{proof}
We now turn to the lower bound and argue again  by contradiction. Assume that  $\alpha={\rm TD}(u;\mu)-\eps>0$ and that there exists a sequence $\{Q_k\}_{k\in \N}\subset  \mathcal{P}^{a.c.}(\R^d)$ such that  $\|x_k\|\to \infty $ with $ x_k \in \partial \varphi_{\eps,k}(u)$ and $\partial \varphi_{\eps,k} $  being any maximal {  cyclical} monotone extension of $\nabla \varphi_{\mu\to P+\eps (Q_k-P)}$ for all $k\in \N$. As $P$ is tight, for every $\beta>0$ there exists $R=R(\beta)>0$ such that $P(R\,\mathbb{B})\geq 1-\beta$.  Now we argue as in the proof of \autoref{lemma:LowerCont} but with the cone $\mathcal{C}_{x_k,e_k,\theta_k}$ containing $R\,\mathbb{B}$ and having vertex $x_k$.    As $\|x_k\|\to \infty$, the angle $\theta_k$ of the cone $\mathcal{C}_{x_k,e_k,\theta_k}$ tends to $0$ and the direction $e_k$ tends (after taking subsequences)  to some $v_0\in \mathcal{S}^{d-1}$. Then we obtain 
\begin{align*}
    ((1-\eps)P+ \eps Q_{k})(\mathcal{C}_{x_k,e_k,\theta_k})&\geq  (1-\eps)P (R\, \mathbb{B})\\
    &\geq (1-({\rm TD}(u;\mu)-\alpha))(1-\beta)> 1-{\rm TD}(u;\mu)+\alpha',
\end{align*}
for some $\beta=\beta({\rm TD}(u;\mu),\alpha) $ small enough and $\alpha'=\alpha'(\beta, \alpha)>0$.
Repeating exactly the same argument of \autoref{lemma:LowerCont} we obtain 
\begin{align*}
     1-{\rm TD}(u;\mu)+\alpha'&<((1-\eps)P+ \eps Q_{k})(\mathcal{C}_{x_k,e_k,\theta_k})\\
     &=\mu((\partial\varphi_{\eps,k})^{-1}(\mathcal{C}_{x_k,e_k,\theta_k})) \leq \mu(\mathcal{A}_{u,e_k,\theta_k}).
\end{align*}
Hence, \autoref{lemma:ConvergenceToOneHalph} yields  $ 1-{\rm TD}(u;\mu)+\alpha' \leq \mu(\{ y: 
      \langle  y-u, v_0\rangle\geq 
0  \} ) $. However, as $\mu\ll \mathcal{L}_d$, it holds that
\begin{align*}
 \mu(\{ y: 
      \langle  y-u, v_0\rangle\geq 
0  \} )  &\leq \max_{v\in \mathcal{S}^{d-1}}\mu(\{ y: 
      \langle  y-u, v\rangle\geq 
0  \} )\\
&=\max_{v\in \mathcal{S}^{d-1}}(1-\mu(\{ y: 
      \langle  y-u, v\rangle\leq  
0  \} )\\
&=1-{\rm TD}(u;\mu),
\end{align*}
which leads to a contradiction. Hence, the following result  follows and the proof of \autoref{Theo:mainCont} finishes here. 
\begin{lemma}
    \label{lemma:LowerCont2}
    Let $u\in \R^d$ be such that ${\rm TD}(u;\mu)>0$. Then for every $\eps\in [0,{\rm TD}(u;\mu))$, there exists $R>0$, such that 
    $$ \{ x\in \partial \varphi_{\mu\to \nu}(u):\ \varphi_{\mu\to \nu} \in \Gamma_\eps(P)\}\subset R\,\mathbb{B}.
 $$
\end{lemma}

\subsection{Empirical optimal transport: proof of \autoref{Theo:mainDiscrete2}}\label{Sect:proofDisc}
For two $n\times n$ real matrices $A=(A_{ij})$ and $B=(B_{ij})$, where $i,j\in [n]$ and $[n]=\{1,\dots, n\}$, let  $\langle A, B \rangle_F=\sum_{i,j\in [n] } A_{ij} B_{ij}$ denote their Frobenius inner product.  
An optimal matching  between the dataset $X^{(n)}$ and the reference points $u^{(n)}$  is any $n\times n$ real matrix $\gamma=(\gamma_{ij})$ solving 
\[
 \inf_{\gamma\in \Pi_n} \langle c, \pi \rangle. 
\]
where $c=(c_{ij})$ is a symmetric $n\times n$  cost matrix  with elements $c_{ij}=(\|u_i-X_j\|^2)$, and $ \Pi_{n}$ denotes the {\it Birkhoff polytope} of doubly stochastic matrices \cite{Brualdi.06} defined as 
$$
  \Pi_{n}=\left\{\pi\in { \R}^{n\times n}:\, \sum_{i=1}^n \pi_{ij}  = 1,   \; \sum_{j=1}^n \pi_{ij}=1, \; \pi_{ij}\geq 0\right\}.
$$ 
 A celebrated theorem due to Birkhoff states that the   {vertices} of   $\Pi_{n}$  is the set of permutation matrices. This implies that an optimal matching can be identified with an {\it optimal permutation} $\hat{\sigma}$ on the set $\{1, \dots, n\}$. Therefore, an empirical optimal transport map 
$\hat{T} =\hat{T}_{X^{(n)}}$  can be described as $\hat{T}_n(u_i)= X_{\hat{\sigma}(i)}$, where $\hat{\sigma}$ is an optimal permutation, which 
in general   might not be unique. However, it is well-known  {(see e.g.~Exercise~2.21 in \cite{Villani.03})} that $\hat{T}$ is monotone, i.e.,
\begin{equation}
    \label{eq:monotoneDiscreteProof}
    \langle u_i-u_j, \hat{T}(u_i)-\hat{T}(u_j)\rangle \geq 0  \quad \text{ for all }i,j\in \{1, \dots, n\}
\end{equation}
for any optimal transport map 
$\hat{T}_n $. The monotonicity property \eqref{eq:monotoneDiscreteProof} is the key ingredient necessary to apply \autoref{lemma:conetocone} in the proof of \autoref{Theo:mainDiscrete2}.  {We show at the same time \autoref{Theo:mainDiscrete2} and \autoref{coro:median_contours_disc}. We denote as $\hat{\mathbb{T}}_{X^{(n)}} $ any maximal {  cyclical} monotone interpolation in the sense of \cite{HallinDelBarrioCuestaAlbertosMatran.21}. That is, $\hat{\mathbb{T}}_{X^{(n)}} $ is the subdifferential of a convex function such that $ \hat{T}_{X^{(n)}}(u_i)\in \hat{\mathbb{T}}_{X^{(n)}} (u_i)$ for all $i\in \{1, \dots, n\}$.  }
The claimed finite sample breakdown point results follow from the  upper and lower bounds that we establish in the following two lemmas. Their proofs  follow along the lines of those of given for \autoref{lemma:LowerCont} and \autoref{lemma:LowerCont2}. 
\begin{lemma}\label{lemma:UpperDisc}  { Fix $u\in \R^d$ and $\ell\in \{1, \dots, n\}$  such that ${\rm TD}(u;u_1, \dots, u_n)<\ell/n$.}
 Let $\ell \in \{1, \dots, n\}$ and  $X^{(n,\ell)}_{k}=\{ X_1 \dots, X_{n-\ell},v_1^{(k)}, \dots, v_\ell^{(k)}\}$  where $\{v_1^{(k)}, \dots, v_\ell^{(k)} \}\subset  k v_u+\mathbb{B}$,  $k\in \N$ for some
 $$v_u\in \argmin_{v\in \mathcal{S}^{d-1}} \left\{ \frac{1}{n} \sum_{i=1}^{n} \mathbf{1}_{\left( \langle v, u_i - u \rangle \geq 0 \right)} \right\}.$$   
Then it holds that  
 $$\lim_{k\to \infty} \inf_{x\in \hat{\mathbb{T}}_{X^{(n,\ell)}_{k}}(u)}\|x \|=\infty.$$ 
\end{lemma}
\begin{proof}
  We argue by contradiction.  {Fix $u\in \R^d$  such that ${\rm TD}(u_i;u^{(n)})<\ell/n$.}    After taking a subsequence, we can assume that   {$\bigcup_{k\in \N}\hat{\mathbb{T}}_{X^{(n,\ell)}_{k}}(u) $} is contained in $R\,\mathbb{B}$ for some $R>0$. Fix  {$x_k\in \hat{\mathbb{T}}_{X^{(n,\ell)}_{k}}(u)$.}  Let $\mathcal{C}_{x_k, e_k,\theta_k}$ be the cone of rays arising from $u $ and  {intersecting} the set $\{v_1^{(k)}, \dots, v_\ell^{(k)}\}$. It holds that 
    \begin{equation}
    \label{eq:ToContradictlowerbound1Disc}
    \nu_{\ell, k}(\mathcal{C}_{x_k,e_k,\theta_k})\geq \frac{\ell}{k} \geq {\rm TD}( {u};u^{(n)})+\alpha
\end{equation}
  for some $\alpha>0$, where $\nu_{\ell, k}=\frac{1}{n}\sum_{x\in X^{(n,\ell)}_{k}}\delta_{x}$ denotes empirical measure of the dataset $X^{(n,\ell)}_{k}$.  Since  {$\hat{\mathbb{T}}_{X^{(n,\ell)}_{k}}$} is monotone, \autoref{lemma:conetocone} implies that
$$(\hat{ {\mathbb{T}}}_{X^{(n,\ell)}_{k}})^{-1}(X^{(n,\ell)}_{k}\cap \mathcal{C}_{x_k,e_k,\theta_k})\subset  \mathcal{A}_{u_i,e_k,\theta_k},  $$
which combined with \eqref{eq:ToContradictlowerbound1Disc} leads to
\begin{align}
    \begin{split}
        \label{eq:ToContradictlowerbound2Disc}
 {\rm TD}( {u};u^{(n)})+\alpha  & {\le} \nu_{\ell,k}(\mathcal{C}_{x_k,e_k,\theta_k}) \\
 &=\mu_n((\hat{ {\mathbb{T}}}_{X^{(n,\ell)}_{k}})^{-1}(X^{(n,\ell)}_{k}\cap \mathcal{C}_{x_k,e_k,\theta_k} ) {)}
 \leq  \mu_n(\mathcal{A}_{ {u},e_k,\theta_k}),
    \end{split}
\end{align}
where $\mu_n=\frac{1}{n}\sum_{i=1}^n\delta_{u_i}$. 
Hence, \autoref{lemma:ConvergenceToOneHalph} leads to a contradiction when taking limits in \eqref{eq:ToContradictlowerbound2Disc} as $k\to \infty$. (Note that, as in \autoref{lemma:LowerCont}, $\theta_k\to 0  $ and $e_k\to v$.) 
\end{proof}
\begin{lemma}
    \label{lemma:LowerDisc}
    Fix $i\in \{ 1, \dots, n\}$ and $u\in \R$.  Then, for every $\ell\in \N$ such that   ${\rm TD}^-( {u};u^{(n)})>\ell/n$, there exists a radius $R>0$ such that $\hat{ {\mathbb{T}}}_{X^{(n,\ell)}_{k}}(u)\subset  R\,\mathbb{B}$ for all $X^{(n,\ell)}_{k}\in \mathcal{Q}_{\ell,n}$.
\end{lemma}
\begin{proof} 
 Let $R'>0$ be such that $X_i\in R'\, \mathbb{B}$ for all $i=1, \dots, n$. 
 {Assume  that there exists a sequence $\{X^{(n,\ell)}_{k}\}_{k\in \N}\subset  \mathcal{Q}_{\ell,n}$ such that $\|x_k\|\to \infty$  as $k\to \infty$ and for some $x_k\in \hat{ {\mathbb{T}}}_{X^{(n,\ell)}_{k}}(u) $.}  We define the counting measures $\nu_{\ell, k}=\frac{1}{n}\sum_{ {v}\in X^{(n,\ell)}_{k}}\delta_{ {v}}$ and $\mu_{n}=\frac{1}{n}\sum_{i=1}^n \delta_{u_i}$. We repeat the same steps of the proof of \autoref{lemma:LowerCont2} but exchanging $P+\eps (Q-P) $ by $\nu_{\ell, k}$ and $\mu$ by $\mu_n$.  {That is, we define the smallest cone $\mathcal{C}_{x_k,e_k,\theta_k}$ with vertex $x_k$ and containing $R'\, \mathbb{B}$. Since,  $\|x_k\|\to \infty$, it follows that $\theta_k\to 0$ and  $e_k\to v_0\in \mathcal{S}^{d-1}$, after taking subsequences. Since the extension $\hat{{\mathbb{T}}}_{X^{(n,\ell)}_{k}}$ is monotone, \autoref{lemma:conetocone} implies that 
$$ \hat{{\mathbb{T}}}_{X^{(n,\ell)}_{k}}^{-1}(\mathcal{C}_{x_k,e_k,\theta_k}) \subset \mathcal{A}_{u,e_k,\theta_k}. $$
We integrate with respect to $\mu_n$ in both sides to get
$$ \mu_n(\mathcal{A}_{u,e_k,\theta_k}) \geq \mu_n(\hat{{\mathbb{T}}}_{X^{(n,\ell)}_{k}}^{-1}(\mathcal{C}_{x_k,e_k,\theta_k}) ) =  \nu_{\ell,n}(\mathcal{C}_{x_k,e_k,\theta_k}) \geq \alpha(\ell,n) , $$
where 
$$ \alpha(\ell,n)=\begin{cases}
    \frac{n+1-\ell}{n} & \text{if $u=u_i$ for some $i\in \{1, \dots, n\}$,}\\
    \frac{n-\ell}{n} & \text{otherwise}.
\end{cases} $$
We take limits as $k\to \infty$ and apply \autoref{lemma:ConvergenceToOneHalph} to get
$$  \alpha(\ell,n) \leq \mu_n (\{ y: 
      \langle  y-{u}, v_0\rangle\geq 
0  \} ) \leq \max_{v\in \mathcal{S}^{d-1}}\mu_n (\{ y: 
      \langle  y-{u}, v\rangle\geq 
0  \} )  .$$
Hence, if $u=u_i$ for some $i\in \{1, \dots, n\}$
$$  \frac{\ell}{n} \geq \frac{n-1}{n}- \max_{v\in \mathcal{S}^{d-1}}\mu_n (\{ y: 
      \langle  y-{u}, v\rangle\geq 
0  \} ) =\frac{1}{n}+\min_{v\in \mathcal{S}^{d-1}}\mu_n (\{ y: 
      \langle  y-{u}, v\rangle> 
0  \} ) ,$$
and, otherwise, 
$$  \frac{\ell}{n} \geq 1- \max_{v\in \mathcal{S}^{d-1}}\mu_n (\{ y: 
      \langle  y-{u}, v\rangle\geq 
0  \} ) =\min_{v\in \mathcal{S}^{d-1}}\mu_n (\{ y: 
      \langle  y-{u}, v\rangle> 
0  \} ) .$$
This concludes the proof.} 
\end{proof}
\subsection{Additional auxiliary results}
\subsubsection{Proof of \autoref{lemma:Consistency-ofTukey}}\label{Section:Proof-of-lemma-Tukey}

 { \begin{proof}[Proof of \autoref{lemma:Consistency-ofTukey}]
 Let $\{x_n\}_{n\in \N}$ be a sequence such that 
 $$  \lim_n |{\rm TD}(x_n; \mu_n) - {\rm TD}(x_n; \mu)|= \lim_n\sup_{x\in \R^d}|{\rm TD}(x; \mu_n) - {\rm TD}(x; \mu)|.$$
 Assume first that $\|x_n\|\to \infty$. Since $\mu_n$ converges to $\mu$ in distribution, for every $\delta>0$, there exists  a compact convex set $K$ with $ \mu_n(K)\geq 1-\delta $ and $ \mu(K)\geq 1-\delta$. Since $\|x_n\|\to \infty$, there exists $n_0\in \N$ such that  $x_n\notin K$ for $n\geq n_0$. Therefore, the separation theorem \cite[Theorem 11.2]{Rockafellar.70} implies there exists a closed halfspace $H_n$ containing $x_n$ with $\mu_n(H_n)\leq  \delta$ and  $\mu(H_n)\leq \delta$,  for all $n\geq n_0$. This implies that 
 $$ \limsup_n |{\rm TD}(x_n; \mu_n) - {\rm TD}(x_n; \mu) | \leq 2\delta, \quad \text{for all }\delta>0,$$
which concludes the proof in the case that $\|x_n\|\to \infty$. We can assume now that $x_n \to x_*$. 
     We define, for $\eps>0$, the continuous and bounded functions $\Theta_\eps, \Phi_\eps:\R^d\times \R^d \times\mathcal{S}^{d-1}\to \R $ as  $$ \Theta_\eps (x,y,s)= \begin{cases}
         1 & \text{if } \langle  s-x, y\rangle\leq -\eps, \\
       -\frac{\langle  s-x, y\rangle}{\eps}  & \text{if } \langle  s-x, y\rangle \in [-\eps,0],\\
         0 &  \text{otherwise},
     \end{cases}  $$
     and 
     $$ \Phi_\eps (x,y,s)= \begin{cases}
         1 & \text{if } \langle  s-x, y\rangle\leq 0, \\
       1-\frac{\langle  s-x, y\rangle}{\eps}  & \text{if } \langle  s-x, y\rangle \in [0,\eps],\\
         0 &  \text{otherwise}.
     \end{cases}  $$
We note that the functions $\Theta_\eps $ and $ \Phi_\eps$ are Lipschitz with constant $L_\eps$ depending exclusively on $\eps$.  
Then if $\mu_n$ converges in distribution to $\mu$, it follows that 
\begin{equation}
    \label{forTheta-phi}
     \sup_{x,y} \left| \int \Theta_\eps(x,y,s)d(\mu_n-\mu)(s) \right| \to 0 \quad {\rm and} \quad \sup_{x,y} \left| \int \Phi_\eps(x,y,s)d(\mu_n-\mu)(s) \right|\to 0. 
\end{equation}
Then 
\begin{align*}
    {\rm TD}(x_n; \mu_n) &\geq \inf_{v} \int \Theta_\eps(x_n,y,s) d \mu_n(s)\\
    &\geq \inf_{y} \int \Theta_\eps(x_n,y,s) d \mu(s)-  \sup_{x,y} \left|\int \Theta_\eps(x,y,s) d (\mu_n-\mu)(s)\right|
\end{align*}
so that, by taking limits, using \eqref{forTheta-phi} and the Lipschitz continuity of $\Theta_\eps$,
$$ \liminf_{n} {\rm TD}(x_n; \mu) \geq    \inf_{y}  \int \Theta_\eps(x_*,y,s) d \mu(s) . $$
Since, for every $\eps> 0$ and $n\in \N$,  
$$ |\Theta_\eps(x_*,y,s)- {\bf 1}_{\{ \langle  s-x_*, y\rangle\leq 0\}} | \leq {\bf 1}_{\{ \langle  s-x_*, y\rangle \in [-\eps,0]\}}, $$
we obtain that for every $\eps>0$
\begin{equation}
    \label{eq:Bound-Tukey-lower-proof-consistency-almost}
    \liminf_{n}  {\rm TD}(x_n; \mu_n) \geq {\rm TD}(x_*; \mu)  -  \sup_{y}\mu(\{s: \langle  s-x_*, y\rangle \in [-\eps,0]\} ).
\end{equation}
We claim that 
\begin{equation}
    \label{eps-to-0}
 L:=  \lim_{
   \eps\to 0
   } \limsup_{n\to \infty} \sup_{y}\mu(\{s: \langle  s-x_*, y\rangle \in [-\eps,0]\} )=0.
\end{equation}
Let $\{\eps_n\}$ and $\{y_n\}$ be such that $\eps_n\to 0$ and 
$$ \liminf_n \mu(\{s: \langle  s-x_*, y_n\rangle \in [-\eps_n,0]\} ) =L. $$
By compactness of the unit sphere, we can assume (after taking subsequences) that $y_n\to y_*$ as $n\to \infty$. Fix $\delta>0$ and let $K=K_\delta\subset \R^d$ be a compact set such that $\mu(K)\geq 1-\delta$. Then we have 
$$ \mu(\{s: \langle  s-x_n, y_n\rangle \in [-\eps,0]\} )  \leq \mu(\{s\in K: \langle  s-x_*, y_n\rangle \in [-\eps_n,0]\} ) +\delta  ,$$
which yields 
$$ \mu(\{s: \langle  s-x_*, y_n\rangle \in [-\eps,0]\} )  \leq \mu(\{s\in K: \langle  s-x_*, y\rangle \in [-(\eps_n+C\|y_n-y_*\| ),0]\} ) +\delta ,$$
where the constant $C$ depends on $x_*$ and $K$. Since $ \mu$ does not give mass to hyperplanes, and 
$\eps_n+C\|y_n-y_*\| \to 0$, 
the dominated convergence theorem yields 
$ \mu(\{s: \langle  s-x_*, y_n\rangle \in [-\eps,0]\} )  \leq \delta .$  Since $\delta>0$ was arbitrarily chosen, \eqref{eps-to-0} follows. Hence, \eqref{eq:Bound-Tukey-lower-proof-consistency-almost} yields 
$$ \liminf_{n}  {\rm TD}(x_n; \mu_n) \geq {\rm TD}(x_*; \mu),   $$
for all $\mu_n \to \mu$ in distribution and $x_n\to x_*$. 
The reverse bound $$  \limsup_{n}  {\rm TD}(x_n; \mu_n) \leq {\rm TD}(x_*; \mu) $$ 
follows the same steps but using the function $\Phi_\eps$ instead of $\Theta_\eps$. This shows that 
$$ \limsup_{n}  {\rm TD}(x_n; \mu_n) =  {\rm TD}(x_*; \mu)=\liminf_{n}  {\rm TD}(x_n; \mu_n) , $$
which in turn implies that
$$ \lim_n \sup_{x\in \R^d}|{\rm TD}(x; \mu_n) - {\rm TD}(x; \mu)|= \lim_n |{\rm TD}(x_n; \mu_n) - {\rm TD}(x_n; \mu)|=0 .$$
This shows the first claim of the lemma. To prove second one, let $x_{\frac{1}{2}}$ be such that ${\rm TD}(x_{\frac{1}{2}}; \mu)=1/2$. Then, 
$$\sup_x {\rm TD}(x; \mu_n) \geq  {\rm TD}(x_{\frac{1}{2}}; \mu_n) \to  1/2, \quad \text{as $n\to \infty$,}$$
and 
$$ \lim_{n} \sup_x {\rm TD}(x; \mu_n) \leq  \sup_x {\rm TD}(x; \mu) + \lim_n \sup_{x\in \R^d}|{\rm TD}(x; \mu_n) - {\rm TD}(x; \mu)| = \frac{1}{2} , $$
which concludes the proof.

 \end{proof}
}

 \subsubsection{Univariate transport-based median}\label{Section:univariate median}

{
The univariate result stated below has been well-known by the community working on transport-based quantiles and is stated below for the sake of completeness. 

\begin{lemma}\label{lem:univariate-median} Fix a reference measure $\mu\in \mathcal{P}^{a.c.}(\R)$. Then, for any  $P\in \mathcal{P}^{a.c.}(\R)$,
$$ m_\pm(P) = \{ x_*\in \R:  P(\{z\in \R: z\leq x_*\}) = P(\{z\in \R: z> x_*\})=1/2 \}.$$
\end{lemma}
\begin{proof}
For $\nu\in \mathcal{P}^{a.c.}(\R)$, we denote as 
    $$ m_\nu = \{ x_*\in \R:  \nu(\{z\in \R: z\leq x_*\}) = \nu(\{z\in \R: z> x_*\})=1/2 \}. $$
Let  $\mathbb{T}_{\mu\to P }$ be any maximal {  cyclical} monotone extension of the unique monotone map ${T}_{\mu\to P }$ pushing $\mu$ forward to $P$.
Our goal is to show that 
\begin{equation}
    \label{eq:Proof-of-median}
m_\pm (P)=\mathbb{T}_{\mu\to P }(m_\mu)= m_P
\end{equation}
Note that by definition $m_\pm (P)=\mathbb{T}_{\mu\to P }(m_\mu)$. Let $u^*\in m_\mu$ be a median of the reference measure $\mu$. 
 Then using the monotonicity of ${T}_{\mu\to P }$, for every $x_*\in \mathbb{T}_{\mu\to P }(u_*) $,   if $\pm u< \pm u_*$ then $ \pm x \leq \pm x_*$ for all $x\in \mathbb{T}_{\mu\to P }(u)$. As a consequence,  for every $x_*\in \mathbb{T}_{\mu\to P }(u_*) $,
    $$ P(\{z\in \R: z< x_*\}) = \mu(\{ u\in \R : {T}_{\mu\to P }(u) <x_*\}) \leq    \mu(\{ u \in \R:u \leq u_*\})  =\frac{1}{2}$$
    and 
      $$ P(\{z\in \R: z> x_*\}) = \mu(\{ u\in \R : {T}_{\mu\to P }(u) >x_*\}) \leq    \mu(\{ u\in \R :u \geq u_*\})  =\frac{1}{2},$$
so that $x_*\in m_P$. Now let  $\mathbb{T}_{\mu\to P }^{-1}$ be the inverse of $\mathbb{T}_{\mu\to P }$ in the set-valued sense, i.e., $\mathbb{T}_{\mu\to P }^{-1}(x) =\{u: x\in  \mathbb{T}_{\mu\to P }(u) \}$.  Note that the transport map  $T_{P\to \mu}$ from  $P$ to $\mu$ is a Borel  measurable selection of $\mathbb{T}_{\mu\to P }^{-1}$ (see \cite[Remark~16]{McCann.95} and \cite[Corollary~23.5.1.]{Rockafellar.70}). Since $\mathbb{T}_{\mu\to P }^{-1}$ is also monotone, we get that if $x_*$ is a median of $P$ then any $u_*\in \mathbb{T}_{\mu\to P }^{-1}(x_*)$ is a median of $ \mu$. Therefore, \eqref{eq:Proof-of-median} holds and the proof is concluded. 
\end{proof}
}


\bibliography{biblio2}

\end{document}